%% file: main.tex
\documentclass[11pt,letterpaper]{article}
\input{preamble}

\title{Learning Operators of Geometry with an Interface Autoencoder}
\author{
Aiqing Zhu\thanks{Department of Mathematics, National University of Singapore
(\href{mailto:zaq@nus.edu.sg}{\texttt{zaq@nus.edu.sg}}).}
\and
Pengzhan Jin\thanks{National Engineering Laboratory for Big Data Analysis and
Applications, Peking University
(\href{mailto:jpz@pku.edu.cn}{\texttt{jpz@pku.edu.cn}}).}
}
\date{}

\begin{document}
\maketitle

\begin{abstract}
\input{sec/abstract}
\end{abstract}



\section{Introduction}
\input{sec/intro}

\section{Problem setup}
\label{sec:setup}
\input{sec/problem}

\section{Approximation theory}
\label{sec:approx}
\input{sec/approx}

\section{Interface Autoencoder}
\label{sec:method}
\input{sec/iae}

\section{Operator learning with the IAE and its approximation theory}
\label{sec:operator-approx}
\input{sec/operator}

\section{Proofs}
\label{sec:proofs}
This section proves Theorems~\ref{thm:mimo-codeapprox} and~\ref{thm:opapprox}.
\subsection{Proof of Theorem~\ref{thm:mimo-codeapprox}}
\label{sec:proof1}
\input{sec/proof1}
\subsection{Proof of Theorem~\ref{thm:opapprox}}
\label{sec:proof-opapprox}
\input{sec/proof2}

\section{Numerical experiments}
\label{sec:numerics}
\input{sec/numerical_example}

\section{Conclusion}
\label{sec:conclusion}
\input{sec/conclusion}

\clearpage
\bibliographystyle{plainnat}
\bibliography{sec/ref}

\clearpage
\appendix
\section*{Numerical implementation details}
\input{sec/numerical_details}

\end{document}

%% file: preamble.tex
\usepackage[T1]{fontenc}
\usepackage[utf8]{inputenc}
\usepackage{lmodern}
\usepackage[margin=1in]{geometry}
\usepackage{mathtools,amssymb,amsthm,bm}
\usepackage{graphicx,booktabs,caption}
\usepackage{microtype}
\usepackage{needspace}
\usepackage[numbers,sort&compress]{natbib}
\usepackage{xurl}
\usepackage[hidelinks]{hyperref}

\hypersetup{
  pdftitle={Learning Operators of Geometry with an Interface Autoencoder},
  pdfauthor={Aiqing Zhu and Pengzhan Jin}
}
\graphicspath{{figures/}}
\mathtoolsset{showonlyrefs=true}
\numberwithin{equation}{section}
\allowdisplaybreaks[1]
\AddToHook{cmd/section/before}{\Needspace{6\baselineskip}}

\theoremstyle{plain}
\newtheorem{theorem}{Theorem}[section]
\newtheorem{lemma}[theorem]{Lemma}
\newtheorem{corollary}[theorem]{Corollary}
\newtheorem{proposition}[theorem]{Proposition}
\newtheorem{assumption}{Assumption}

\theoremstyle{definition}

\theoremstyle{remark}

\newcommand{\R}{\mathbb{R}}
\newcommand{\N}{\mathbb{N}}
\newcommand{\Bx}{\mathcal{B}}
\newcommand{\norm}[1]{\lVert #1 \rVert}
\newcommand{\abs}[1]{\lvert #1 \rvert}
\newcommand{\inner}[2]{\langle #1,\, #2\rangle}
\newcommand{\relL}{\ensuremath{\mathrm{rel}\text{-}\ell^2}}

\DeclareMathOperator{\sdf}{sdf}
\DeclareMathOperator*{\argmin}{arg\,min}
\DeclareMathOperator{\dist}{dist}
\DeclareMathOperator{\Ext}{Ext}
\DeclareMathOperator{\reach}{reach}
\newcommand{\bi}{\bm{i}}
\newcommand{\bp}{\bm{p}}

%% file: sec/abstract.tex
Geometry-dependent PDEs define operators whose inputs and outputs may each consist of an oriented interface and a function on that interface, whereas most existing neural operators are formulated on fixed domains.
In this paper, we first establish an approximation theorem for continuous operators between general state sets.
We then instantiate the theorem with the Interface Autoencoder (IAE), which represents interface--function states in a fixed box using finite projection codes and decodes them by zero-set extraction and restriction.
For tensor cosine codes, we prove a uniform $C^0$ reconstruction rate and tube-local gradient convergence.
These estimates yield Hausdorff and graph-Hausdorff error bounds, together with a reference-based topology certificate.
The IAE accommodates both cross-space geometric operators, mapping domain geometries to the associated PDE solution fields, and within-space geometric operators, mapping between interface--function states.
Its effectiveness is demonstrated through numerical experiments on the Poisson equation, interface merging in Hele--Shaw flow, and two-phase Stokes flow with surfactant.

%% file: sec/intro.tex
Many scientific-computing tasks require repeated solutions of partial differential equations indexed by coefficients, data, or geometry.
Operator learning addresses this many-query setting by approximating solution maps, with existing methods largely developed for functions on fixed domains and supported by universal approximation theory \cite{Chen1995Universal,Gupta2021Multiwavelet,He2024MgNO,Jin2022MIONet,Kovachki2021On,Kovachki2023Neural,Li2021Fourier,Lu2021Learning}.
However, many problems of interest are geometry-dependent.
Examples include structural shape and topology optimization \cite{Allaire2004Structural,Bendsoe1988Generating} and fluid--structure interaction \cite{Donea1982Arbitrary,Peskin2002Immersed}.
Evolving geometries also arise in phase-change problems \cite{Karma1998Quantitative,Voller1987Fixed}, Hele--Shaw flows \cite{Hou1994Removing,Saffman1958Penetration}, multiphase flows with surfactant \cite{Stone1990Effects,Xu2006Level}, and PDEs posed on evolving surfaces \cite{Dziuk2007Finite,Dziuk2013Finite}.
In these settings, the domain or interface may itself be an input or output, so the states of the target solution operators may not belong to a common fixed function space.
We distinguish two cases in this paper: \emph{cross-space} operators map domain geometries to the corresponding PDE solution fields, whereas \emph{within-space} operators map one geometric state to another.
In either case, a geometric state may consist of an interface together with a function defined on it.

Existing approaches to variable-domain operator learning broadly follow two strategies: transporting data to a fixed reference domain or representing geometry directly in a common ambient domain. Geo-FNO and DIMON use learned or prescribed coordinate maps \cite{Li2023Fourier,Yin2024Scalable}; D2D and D2E use pullback or extension representations \cite{Xiao2026Deformation}; and FBNO predicts a diffeomorphism from a fixed reference domain together with a reference field \cite{Long2026Deep}. Such approaches become restrictive when the representation relies on a homeomorphic or diffeomorphic correspondence with a single reference domain, since that correspondence cannot be maintained across topology changes.
Ambient-domain methods avoid this restriction by encoding geometry directly, for example through a domain indicator in DAFNO \cite{Liu2023Domain}, signed-distance and point-cloud representations in GINO \cite{Li2023Geometry} and Geom-DeepONet \cite{He2024Geom}, point-cloud integration in PCNO \cite{Zeng2025PointCloud}, physics-attention in Transolver \cite{Luo2025TransolverPlus,Wu2024Transolver,zhou2026transolver}, or coordinate-based neural fields in CORAL \cite{Serrano2023Operator} and discretization-independent surrogates \cite{Duvall2025Discretization}. In most of these approaches, however, geometry is provided as part of the input representation, while the learned operator predicts fields on the prescribed geometry rather than geometric states themselves.
The Manifold Function Encoder (MFE) takes a step toward this broader setting by providing an injective ambient representation of manifold--function pairs through manifold-supported measures and finite moments \cite{Hu2025Manifold}. Its dual-norm reconstruction bound, however, does not by itself yield a stable decoder back to the physical manifold--function state, nor does it certify the topology of a reconstructed interface. Together, these limitations motivate two requirements: a general approximation principle for continuous operators between geometric state sets, and a stable physical decoder for jointly predicted interface--function states.

In this paper, we first establish a universal approximation theorem for continuous operators between general state sets, departing from existing results relying on Schauder bases of Banach spaces \cite{Chen1995Universal,Jin2022MIONet,Lu2021Learning}.
We then introduce the \emph{Interface Autoencoder} (IAE) for oriented interface--function states as a concrete instance of this approximation theorem.
The IAE combines an injective fixed-box representation based on the signed-distance field \cite{Osher1988Fronts} and the midpoint McShane--Whitney extension \cite{McShane1934Extension} with finite projection codes and a zero-set-based physical decoder.
For tensor cosine codes, classical Jackson--Lebesgue estimates provide a uniform $C^0$ reconstruction rate.
We further establish a quantitative local gradient error bound for the unfiltered tensor cosine projection of a signed distance function (Proposition~\ref{prop:approx}(ii)).
Under suitable geometric regularity, these estimates yield Hausdorff and graph-Hausdorff error bounds and a reference-based topology certificate ensuring that the decoded interface is a unique $C^1$ normal graph over the target.

To validate the effectiveness of the IAE, we perform several numerical experiments across two settings. For the cross-space case, the IAE shows superior representation capability on a variable-domain Poisson problem. For the within-space case, IAE-based operators successfully handle complex physical dynamics, including topological mergers in Hele--Shaw flow and coupled interface-surfactant states in two-phase Stokes flow.

%% file: sec/problem.tex
Fix $d\ge2$ and let $\Bx=[0,1]^d$.
We consider the state space
\begin{equation}\label{eq:ifspace}
\mathcal X:=\left\{(\Gamma,f):
\begin{aligned}
&\Gamma=\partial\Omega_\Gamma\text{ for a unique nonempty open set }
\Omega_\Gamma\Subset\operatorname{int}\Bx\\
&\text{with Lipschitz boundary},\qquad
f:\Gamma\to\R,\quad [f]_{\mathrm{Lip}(\Gamma)}<\infty
\end{aligned}
\right\}.
\end{equation}
The Lipschitz seminorm is computed using the Euclidean distance:
\begin{equation}\label{eq:lipseminorm}
[f]_{\mathrm{Lip}(\Gamma)}
:=\sup_{\substack{x,y\in\Gamma\ x\ne y}}
{\abs{f(x)-f(y)}}/{\norm{x-y}}.
\end{equation}
The set $\Omega_\Gamma$ specifies the interior phase and hence the orientation of the interface $\Gamma$.
No connectedness assumption is imposed on $\Omega_\Gamma$ or $\Gamma$.

The operators considered below admit $\mathcal X$ as a component of the input and/or output spaces, possibly with other variables.
A canonical within-space example is
\begin{equation}\label{eq:target}
\mathcal{G}:\mathcal{X}\longrightarrow\mathcal{X},\qquad
(\Gamma_0,f_0)\longmapsto(\Gamma_1,f_1),
\end{equation}
which is learned from samples $\{((\Gamma_0^s,f_0^s),(\Gamma_1^s,f_1^s))\}_{s=1}^S$ satisfying $(\Gamma_1^s,f_1^s)=\mathcal G(\Gamma_0^s,f_0^s)$.
More generally, the domain or codomain of $\mathcal G$ may contain additional product factors, and in a cross-space problem the interface-state factor $\mathcal X$ on one side may be replaced by a different state space, such as a space of full-domain fields.

%% file: sec/approx.tex
Existing operator-approximation results typically assume that the input and output spaces are Banach spaces equipped with Schauder bases.
To accommodate geometric state spaces, we instead establish a general approximation principle for continuous operators between arbitrary state sets.

Let $p,q$ be positive integers, and let $X_i$ and $Y_j$ be sets for $i=1,\ldots,p$ and $j=1,\ldots,q$.
We make the following assumptions:
\begin{assumption}\label{ass:embedding}
There are normed spaces and injective maps
\[
(Z_i^X,\norm{\cdot}_{Z_i^X}),\quad (Z_j^Y,\norm{\cdot}_{Z_j^Y}),
\qquad
\iota_i^X:X_i\to Z_i^X,
\quad
\iota_j^Y:Y_j\to Z_j^Y.
\]
Equip $X_i$ and $Y_j$ with the induced metrics
\[
d_i^X(x,x')=\norm{\iota_i^X(x)-\iota_i^X(x')}_{Z_i^X},
\qquad
d_j^Y(y,y')=\norm{\iota_j^Y(y)-\iota_j^Y(y')}_{Z_j^Y},
\]
and their products with the corresponding max metrics.
\end{assumption}

\begin{assumption}\label{ass:finite-reconstruction}
For $n,m\in\N$, there are $a_{i,n},b_{j,m}\in\N$ and encoding and reconstruction maps
\[
\begin{aligned}
\Phi_{i,n}^X: X_i\to\R^{a_{i,n}},\
\Psi_{i,n}^X: \R^{a_{i,n}}\to Z_i^X,\
\Phi_{j,m}^Y: Y_j\to\R^{b_{j,m}},\
\Psi_{j,m}^Y: \R^{b_{j,m}}\to Z_j^Y,
\end{aligned}
\]
such that $\Psi_{i,n}^X$ and $\Phi_{j,m}^Y$ are continuous, while $\Psi_{j,m}^Y$ is continuous at every point of $\Phi_{j,m}^Y(Y_j)$, for all $i,j,n,m$.
For every nonempty compact set $C\subset X_i$ and $D\subset Y_j$, the reconstructions satisfy
\begin{equation*}
\begin{aligned}
&\lim_{n\to\infty}\sup_{x_i\in C}
\norm{\Psi_{i,n}^X \circ \Phi_{i,n}^X(x_i)-\iota_i^X(x_i)}_{Z_i^X}=0,
\\
&\lim_{m\to\infty}\sup_{y_j\in D}
\norm{\Psi_{j,m}^Y \circ \Phi_{j,m}^Y(y_j)-\iota_j^Y(y_j)}_{Z_j^Y}=0.
\end{aligned}
\end{equation*}

\end{assumption}

\begin{theorem}
\label{thm:mimo-codeapprox}
Under Assumptions~\ref{ass:embedding}--\ref{ass:finite-reconstruction}, let $K_i\subset X_i$ be nonempty and compact, set $\bm K=\prod_{i=1}^p K_i$, and let
$
G=(G_1,\ldots,G_q):\bm K\to\prod_{j=1}^qY_j
$
be continuous.
For every $\varepsilon>0$, there exist $\bm n=(n_1,\ldots,n_p)$, $\bm m=(m_1,\ldots,m_q)$, and a continuous map
\begin{equation}\label{eq:mimo-synthop}
F_{\bm n,\bm m}:\prod_{i=1}^p\R^{a_{i,n_i}}
\longrightarrow\prod_{j=1}^q\R^{b_{j,m_j}},
\qquad
F_{\bm n,\bm m}=(F_{\bm n,\bm m}^{(1)},\ldots,F_{\bm n,\bm m}^{(q)}),
\end{equation}
such that, writing $\bm x=(x_1,\ldots,x_p)$,
\begin{equation}\label{eq:mimo-synthopapprox}
\sup_{\bm x\in\bm K}\max_{1\le j\le q}
\norm{
\Psi_{j,m_j}^Y  \circ F_{\bm n,\bm m}^{(j)}
\bigl((\Phi_{i,n_i}^X(x_i))_{i=1}^p\bigr)
-\iota_j^Y(G_j(\bm x))
}_{Z_j^Y}<\varepsilon.
\end{equation}
\end{theorem}

The construction follows the finite-cover argument of \cite[Lemma~2.9 and Appendix~B]{Jin2022MIONet}. In contrast to Theorem~2.5 of that work, it does not rely on a Schauder basis or canonical projections.
The detailed proof is provided in Section~\ref{sec:proof1}.

Theorem~\ref{thm:mimo-codeapprox} also covers the existing approximation theory of operator learning \cite{Chen1995Universal,Jin2022MIONet,Lu2021Learning}, by viewing an operator \(G:f\mapsto g\) through its associated evaluation map
\[
\bar{G}: \ (f,z)\longmapsto G(f)(z).
\]
For Banach input spaces admitting Schauder bases, canonical truncation and synthesis satisfy the input reconstruction assumption on compact sets.
The query variable $z$ and scalar output use identity embeddings, encoders, and reconstructors.
Thus the theorem retains this familiar construction while allowing other prescribed representations and reconstruction maps.

The manifold function encoding (MFE) of \cite{Hu2025Manifold} provides a particular realization of the assumptions in Theorem~\ref{thm:mimo-codeapprox}.
For a geometric state $x=(M,f_M)$, where $M\subset V:=[0,1]^d$ is a compact $k$-dimensional Lipschitz manifold, $0\le k\le d$, and $f_M\in L^2(M,\mathcal H^k)$, its MFE representation is
\[
\iota_0(x)=(\mu_M,\mu_{M,f})\in ((C(V))')^2,\
\mu_M(v)=\int_Mv\,d\mathcal H^k,\
\mu_{M,f}(v)=\int_Mf_Mv\,d\mathcal H^k.
\]
MFE then constructs finite-dimensional encodings by evaluating these distributions using test functions:
\[
\begin{aligned}
\Phi_n^{\mathrm{MFE}}(M,f_M)
&=\left(
\left(\int_M\phi_{n,\ell}\,d\mathcal H^k\right)_{\ell=1}^{a_n},
\left(\int_Mf_M\phi_{n,\ell}\,d\mathcal H^k\right)_{\ell=1}^{a_n}
\right).
\end{aligned}
\]
Here $\{\phi_{n,\ell}\}_{\ell=1}^{a_n}$ are linearly independent basis functions in $C(V)$.
Under an approximation assumption on the chosen test-function spaces, these moments admit a uniform reconstruction in the dual space.
Thus, the MFE provides the normed-space embedding and finite-dimensional encoding scheme required by Theorem~\ref{thm:mimo-codeapprox}.
However, its reconstruction does not provide, in general, a stable inverse from the embedded representation back to the original geometric state $(M,f_M)$.

%% file: sec/iae.tex
The Interface Autoencoder (IAE) developed below provides another realization of the approximation framework in Theorem~\ref{thm:mimo-codeapprox}, while additionally enabling stable decoding back to the original interface--function state space.
We then construct the encoder decoder pair in this section.

For $(\Gamma,f)\in\mathcal X$, we first extend the geometry and the interface function to the fixed domain $\Bx$.
Define the signed distance function $u_\Gamma=\sdf_\Gamma:\Bx\to\R$ by
\begin{equation}\label{eq:sdf}
\sdf_\Gamma(\xi)=
\begin{cases}
-\dist(\xi,\Gamma),&\xi\in\Omega_\Gamma,\\
0,&\xi\in\Gamma,\\
\dist(\xi,\Gamma),&\xi\in\Bx\setminus\overline{\Omega_\Gamma},
\end{cases}
\qquad
\dist(\xi,\Gamma)=\inf_{y\in\Gamma}\norm{\xi-y}.
\end{equation}
Then $\Gamma=u_\Gamma^{-1}(0)$ and $[u_\Gamma]_{\mathrm{Lip}(\Bx)}\le1$.

Let $L_f=[f]_{\mathrm{Lip}(\Gamma)}$ and define $u_f=\Ext_{L_f}^{\mathrm{mid}}[f]$, where, for $L\ge L_f$,
\begin{equation}\label{eq:mcshane}
\Ext_L^{\mathrm{mid}}[f](\xi)=\tfrac12\Big(\inf_{y\in\Gamma}\big(f(y)+L\norm{\xi-y}\big)+\sup_{y\in\Gamma}\big(f(y)-L\norm{\xi-y}\big)\Big).
\end{equation}
The extension satisfies $\Ext_L^{\mathrm{mid}}[f]|_\Gamma=f$ and $[\Ext_L^{\mathrm{mid}}[f]]_{\mathrm{Lip}(\Bx)}\le L$~\cite{McShane1934Extension}.
Hence, the map $(\Gamma,f)\mapsto (u_\Gamma,u_f)$ provides an injective embedding of $\mathcal X$ into a normed function space on $\Bx$.

Let $\{\phi_m\}_{m\ge0}\subset C^1([0,1];\R)$ be a complete orthonormal family in $L^2(0,1)$.
Write $\N_0=\{0,1,\ldots\}$ and set $\phi_{\bi}(\xi)=\prod_{l=1}^d\phi_{i_l}(\xi_l)$ for $\bi\in\N_0^d$.
For $r\in\N$, set $I_r=\{0,\dots,r-1\}^d$ and $V_r=\operatorname{span}\{\phi_{\bi}:\bi\in I_r\}$, and define $\mathcal A_r:L^2(\Bx)\to\R^{r^d}$ and $\mathcal S_r:\R^{r^d}\to V_r$ by
\begin{equation}\label{eq:projexact}
\mathcal{A}_r u=(c_{\bi})_{\bi\in I_r},\quad c_{\bi}=\int_\Bx u\,\phi_{\bi}\,d\xi,\qquad
\mathcal{S}_r c=\sum_{\bi\in I_r}c_{\bi}\phi_{\bi},\qquad \mathcal{Q}_r u=\mathcal{S}_r\mathcal{A}_r u.
\end{equation}
Thus $\mathcal Q_r$ is the $L^2$-orthogonal projector onto $V_r$.
By completeness and Parseval's identity, $\mathcal Q_ru\to u$ in $L^2(\Bx)$.
Because the orthogonal projectors are uniformly bounded, this convergence is uniform on compact subsets of $L^2(\Bx)$, as required by Theorem~\ref{thm:mimo-codeapprox} when the ambient norm is $L^2$.

The interface encoder is defined by applying $\mathcal A_r$ to the two extended fields:
\begin{equation}\label{eq:encoder}
\mathcal E_r:\mathcal X\longrightarrow\R^{2r^d},\qquad
\mathcal E_r(\Gamma,f)=\big(z_{\mathrm{geom}},z_f\big)
:=\big(\mathcal{A}_r u_\Gamma,\mathcal{A}_r u_f\big).
\end{equation}
For $z=(z_{\mathrm{geom}},z_f)$, synthesis reconstructs the two fields on $\Bx$ and then the interface pair:
\begin{equation}\label{eq:decode}
\begin{gathered}
(\widetilde u_\Gamma,\widetilde u_f)
=\big(\mathcal S_r z_{\mathrm{geom}},\mathcal S_r z_f\big), \qquad
\widetilde\Gamma=\{\xi\in\Bx:\widetilde u_\Gamma(\xi)=0\},\qquad
\widetilde f=\widetilde u_f|_{\widetilde\Gamma}.
\end{gathered}
\end{equation}
Not every code yields an admissible interface state.
Define
\[
\mathcal Z_r^{\mathrm{adm}}
:=\left\{z=(z_{\mathrm{geom}},z_f)\in\R^{2r^d}:
(\widetilde\Gamma,\widetilde f)\in\mathcal X,
\quad \Omega_{\widetilde\Gamma}=\{\widetilde u_\Gamma<0\}
\right\}.
\]
The interface decoder is the partial map
\begin{equation}\label{eq:decoder}
\mathcal D_r:\mathcal Z_r^{\mathrm{adm}}\longrightarrow\mathcal X,\qquad
\big(z_{\mathrm{geom}},z_f\big)\longmapsto\big(\widetilde\Gamma,\widetilde f\big).
\end{equation}
We call the resulting encoder decoder pair the Interface Autoencoder (IAE).

The tensor-product cosine basis used below is
\begin{equation}\label{eq:cosbasis}
\phi_0^{\mathrm{cos}}(x)=1,\quad
\phi_m^{\mathrm{cos}}(x)=\sqrt2\cos(\pi mx)\quad(m\ge1),\quad
\phi_{\bi}^{{\mathrm{cos}}}(\xi)=\prod_{l=1}^d\phi_{i_l}^{{\mathrm{cos}}}(\xi_l).
\end{equation}
This choice is motivated by both the implementation and the analysis.
First, the tensor-product cosine representation admits efficient coefficient transforms through the discrete cosine transform (DCT), substantially reducing the cost of encoding and reconstruction~\cite{Shen2011Spectral}.
Second, its regularity and spectral approximation properties provide the estimates required by the reconstruction theory developed below.

DeepSDF and SIREN have been used to approximate signed-distance fields with neural networks \cite{Park2019DeepSDF,Sitzmann2020Implicit}.
The IAE represents geometry by the projection coefficients of its signed-distance field in a prescribed orthonormal basis.
The resulting codes provide coordinates for learning maps between geometric states, separating the construction of the representation from the learning of the operator.

%% file: sec/operator.tex
We then apply the IAE to operator learning between geometric state spaces.

Equip $C(\Bx)^2$ with the max norm and define
\[
\iota:\mathcal X\longrightarrow C(\Bx)^2,\qquad
\iota(\Gamma,f)=(u_\Gamma,u_f),
\qquad
\mathcal R(u,v):=\bigl(u^{-1}(0),v|_{u^{-1}(0)}\bigr).
\]
The partial map $\mathcal R$ is defined whenever its output belongs to $\mathcal X$ and has interior $\{u<0\}$.
Equations~\eqref{eq:sdf} and \eqref{eq:mcshane} imply $\mathcal R\circ\iota=\operatorname{Id}_{\mathcal X}$, so $\iota$ is injective.
Let $K\subset\mathcal X$ be nonempty and compact under the pullback metric induced by $\iota$, and suppose that $\mathcal G|_K$ is continuous.
Then
\[
\overline{\mathcal G}:=\iota\circ\mathcal G:K\longrightarrow C(\Bx)^2,
\qquad
\mathcal G|_K=\mathcal R\circ\overline{\mathcal G}.
\]
The finite codec satisfies $\mathcal E_r=(\mathcal A_r\times\mathcal A_r)\circ\iota$ and $\mathcal D_r=\mathcal R\circ(\mathcal S_r\times\mathcal S_r)$ on $\mathcal Z_r^{\mathrm{adm}}$.
Proposition~\ref{prop:approx}(i) provides uniform reconstruction, which verifies \textup{(\ref{ass:finite-reconstruction})}.
Applying Theorem~\ref{thm:mimo-codeapprox} with $p=q=1$ yields the following corollary.

\begin{corollary}\label{cor:iae-code}
Let $K\subset\mathcal X$ be nonempty and compact in the pullback metric induced by $\iota$, let $\mathcal G:K\to\mathcal X$ be continuous with respect to the pullback metrics, and assume $[f]_{\mathrm{Lip}(\Gamma)}\le\Lambda$ for every $(\Gamma,f)\in K\cup\mathcal G(K)$.
Equip the IAE with the tensor cosine codec.
Then for every $\varepsilon>0$ there exist orders $r_{\mathrm{in}},r_{\mathrm{out}}\ge2$ and a continuous map $F:\R^{2r_{\mathrm{in}}^d}\to\R^{2r_{\mathrm{out}}^d}$ such that the synthesized fields approximate the embedded target uniformly,
\begin{equation}\label{eq:iae-code}
\sup_{x\in K}\bigl\|(\mathcal S_{r_{\mathrm{out}}}\!\times\mathcal S_{r_{\mathrm{out}}})\,
 F(\mathcal E_{r_{\mathrm{in}}}(x))-\iota(\mathcal G(x))\bigr\|_{C(\Bx)^2}<\varepsilon.
\end{equation}
\end{corollary}

Corollary~\ref{cor:iae-code} controls the ambient max-norm error of the reconstructed signed-distance and surface-extension fields.
To translate this ambient accuracy into errors between the corresponding interface--function states, we first introduce the relevant metrics.
For nonempty compact subsets $K,L$ of a metric space $(M,\mathsf d)$, let
\begin{equation}\label{eq:hausdorff}
d_H^{\mathsf d}(K,L):=\max\!\left\{
\sup_{x\in K}\inf_{y\in L}\mathsf d(x,y),
\sup_{y\in L}\inf_{x\in K}\mathsf d(x,y)\right\}.
\end{equation}
We write $d_H$ when $\mathsf d$ is the Euclidean metric and, for nonempty $A,B\subset\R^d$, set $\dist(A,B):=\inf_{a\in A,\,b\in B}\norm{a-b}$.
On $\Bx\times\R$, let
\[
d_\times((x,s),(y,t))=\max\{\norm{x-y},\abs{s-t}\},
\qquad d_H^\times:=d_H^{d_\times},
\]
and, for $g:N\to\R$, let $\operatorname{gr}_N g:=\{(x,g(x)):x\in N\}$.
Every state has a compact graph, and the map $(\Gamma,f)\mapsto\operatorname{gr}_\Gamma f$ is injective.
Hence, for states $(\Gamma,f)$ and $(\Gamma',f')$,
\begin{equation}\label{eq:graphmetric}
d_{\mathrm{gr}}\big((\Gamma,f),(\Gamma',f')\big)
:=d_H^\times\big(\operatorname{gr}_\Gamma f,\operatorname{gr}_{\Gamma'} f'\big)
\end{equation}
defines a metric on $\mathcal X$.
In particular, $d_{\mathrm{gr}}((\Gamma,0),(\Gamma',0))=d_H(\Gamma,\Gamma')$. For an interface $\Gamma$, we define its reach by
\begin{equation*}
\reach(\Gamma):=\sup\left\{s\ge0: 
\forall\,\xi\in\R^d,\ \dist(\xi,\Gamma)<s\ \Longrightarrow 
\exists!\,y\in\Gamma,\norm{\xi-y}=\dist(\xi,\Gamma) 
\right\},
\end{equation*}
and the corresponding nearest-point projection by
\begin{equation*}
\pi_\Gamma:\{\xi\in\R^d:\dist(\xi,\Gamma)<\reach(\Gamma)\}\longrightarrow\Gamma,\
\pi_\Gamma(\xi):=\argmin_{y\in\Gamma}\norm{\xi-y}.
\end{equation*}

With these definitions, we can now state the main result of this paper, which  translates ambient field accuracy into the original interface--function topology: Hausdorff, graph-Hausdorff, and transported $L^p$ error bounds for the decoded interface--function states, together with a topology certificate.
Its proof, including the required decoding and cosine estimates, is given in Section~\ref{sec:proof-opapprox}.

\begin{theorem}\label{thm:opapprox}
Let $K\subset\mathcal X$ be nonempty and compact in the pullback metric induced by $\iota$, and let $\mathcal G:K\to\mathcal X$ be continuous with respect to the pullback metrics.
Assume that there are constants $\Lambda<\infty$ and $a>0$ such that:
\begin{enumerate}
\item[(H1)] $[f]_{\mathrm{Lip}(\Gamma)}\le\Lambda$ for every state $(\Gamma,f)\in K\cup\mathcal G(K)$;
\item[(H2)] for every $(\Gamma_1,f_1)\in\mathcal G(K)$, the interface $\Gamma_1$ is of class $C^2$, with $\reach(\Gamma_1)\ge a$ and $\dist(\Gamma_1,\partial\Bx)\ge a$.
\end{enumerate}
Equip the IAE with the tensor cosine codec \eqref{eq:cosbasis}.
Then, for every $\varepsilon\in(0,a)$, there exist orders $r_{\mathrm{in}},r_{\mathrm{out}}\ge2$, and a continuous map $F:\R^{2r_{\mathrm{in}}^d}\to\R^{2r_{\mathrm{out}}^d}$ such that the operator $\widehat{\mathcal G}:=\mathcal D_{r_{\mathrm{out}}}\circ F\circ \mathcal E_{r_{\mathrm{in}}}$ satisfies $ \widehat{\mathcal G}(K)\subset \mathcal X$, and for every $x\in K$, writing $\mathcal G(x)=(\Gamma_1,f_1)$ and $\widehat{\mathcal G}(x)=(\widetilde\Gamma_1,\widetilde f_1)$:
\begin{enumerate}
\item\label{it:op-topo} There is a unique $h\in C^1(\Gamma_1)$ with $\norm{h}_{L^\infty(\Gamma_1)}<\varepsilon$ such that $\widetilde\Gamma_1=\{y+h(y)\nu_{\Gamma_1}(y):y\in\Gamma_1\}$, where $\nu_{\Gamma_1}$ is the outward unit normal of $\Omega_{\Gamma_1}$; in particular, $\pi_{\Gamma_1}|_{\widetilde\Gamma_1}$ is a $C^1$ diffeomorphism, so $\widetilde\Gamma_1$ is $C^1$-diffeomorphic to $\Gamma_1$.
\item\label{it:op-acc} With $q=(\pi_{\Gamma_1}|_{\widetilde\Gamma_1})^{-1}$, for every $p\in[1,\infty)$,
\[
d_H(\widetilde\Gamma_1,\Gamma_1)\le\varepsilon,
\quad
d_{\mathrm{gr}}\bigl(\widehat{\mathcal G}(x),\mathcal G(x)\bigr)
\le\varepsilon,
\quad
\norm{\widetilde f_1\circ q-f_1}_{L^p(\Gamma_1)}
\le\abs{\Gamma_1}^{1/p}\varepsilon.
\]
\end{enumerate}
Moreover, there is $\rho_0>0$ such that every continuous $F':\R^{2r_{\mathrm{in}}^d}\to\R^{2r_{\mathrm{out}}^d}$ with $\sup_{z\in\mathcal E_{r_{\mathrm{in}}}(K)}\norm{F'(z)-F(z)}_2<\rho_0$ satisfies the same conclusions.
\end{theorem}

Theorem~\ref{thm:opapprox} concerns within-space operators, whose predictions must themselves be decoded into interface--function states.
For a cross-space operator whose output is a full-domain field, the physical decoding step and the geometric hypothesis (H2) are unnecessary, and Theorem~\ref{thm:mimo-codeapprox} alone provides the ambient guarantee.
Geometry-only outputs correspond to $f_1\equiv0$, for which the graph bound reduces to the Hausdorff bound.
By universal approximation, a neural network $F_\theta^{\mathrm{state}}:\R^{2r_{\mathrm{in}}^d}\to\R^{2r_{\mathrm{out}}^d}$ can be chosen to approximate $F$.
Whenever this uniform approximation error is smaller than $\rho_0$, the resulting predictor inherits the conclusions of Theorem~\ref{thm:opapprox}.
In practice, the network parameters $\theta$ are fitted from training data.
The predictor is defined by
\begin{equation}\label{eq:opmap}
\widehat{\mathcal G}_\theta :=\mathcal D_{r_{\mathrm{out}}}\circ F_\theta^{\mathrm{state}} \circ\mathcal E_{r_{\mathrm{in}}}.
\end{equation}

%% file: sec/proof1.tex
\begin{proof}

Fix $\varepsilon>0$.
We first choose $m_j$ such that, for every $j$,
\[
\sup_{y\in G_j(\bm K)}
\norm{\Psi_{j,m_j}^Y\Phi_{j,m_j}^Y(y)-\iota_j^Y(y)}_{Z_j^Y}
<\frac{\varepsilon}{3}.
\]
Set $\mathcal H_j=\Phi_{j,m_j}^Y\circ G_j$ and $\mathcal C_j=\mathcal H_j(\bm K)$.
Since each $\mathcal C_j$ is compact and $\Psi_{j,m_j}^Y$ is continuous at every point of $\mathcal C_j$, there exists a constant $\eta>0$ such that, for all $j$,
\begin{equation}\label{eq:decoder-local-continuity}
c\in\mathcal C_j,\quad \norm{e-c}_2<\eta
\quad\Longrightarrow\quad
\norm{\Psi_{j,m_j}^Y(e)-\Psi_{j,m_j}^Y(c)}_{Z_j^Y}
<\frac{\varepsilon}{3}
\end{equation}
Write $\mathcal H=(\mathcal H_1,\ldots,\mathcal H_q)$, taking values in $\prod_{j=1}^q\R^{b_{j,m_j}}$ equipped with the max norm.
Set $\widetilde K_i=\iota_i^X(K_i)$.
Since each $\iota_i^X$ is an isometry, the map $(\iota_i^X(x_i))_{i=1}^p\mapsto\mathcal H(\bm x)$ is continuous on $\prod_i\widetilde K_i$.
Applying \cite[Lemma~2.9]{Jin2022MIONet} in the Banach completions of $Z_i^X$, we obtain
a continuous map $T=(T_1,\ldots,T_q):\prod_i Z_i^X\to\prod_j\R^{b_{j,m_j}}$ such that
\begin{equation}\label{eq:tensor-on-embedded-inputs}
\sup_{\bm x\in\bm K}\max_j
\norm{T_j((\iota_i^X(x_i))_{i=1}^p)-\mathcal H_j(\bm x)}_2\leq\frac{\eta}{3}.
\end{equation}
{By the continuity of $T$ and the compactness of $\prod_i\iota_i^X(K_i)$,} we can choose $\delta_{\mathrm{rec}}>0$ such that for $\bm x\in\bm K$, $z_i\in Z_i^X$, $i=1,\ldots,p$, if $\max_i\norm{z_i-\iota_i^X(x_i)}_{Z_i^X}<\delta_{\mathrm{rec}}$, then
\begin{equation}\label{eq:tensor-input-perturbation}
\max_j\norm{T_j((z_i)_{i=1}^p)-T_j((\iota_i^X(x_i))_{i=1}^p)}_2
<\frac{\eta}{3}.
\end{equation}
Choose $n_i$ so that the $i$th input reconstruction error on $K_i$ is smaller than $\delta_{\mathrm{rec}}$, and define, on the entire product of input code spaces, $F_{\bm n,\bm m}((e_i)_{i=1}^p) =T((\Psi_{i,n_i}^X(e_i))_{i=1}^p)$.

This map is continuous.
By the estimates \eqref{eq:tensor-on-embedded-inputs} and \eqref{eq:tensor-input-perturbation}, we have
\[
\sup_{\bm x\in\bm K}\max_j
\norm{F_{\bm n,\bm m}^{(j)}
((\Phi_{i,n_i}^X(x_i))_{i=1}^p)-\mathcal H_j(\bm x)}_2\leq\frac{2\eta}{3}<\eta.
\]
Using \eqref{eq:decoder-local-continuity} and the chosen output reconstruction errors, we obtain
\begin{equation}\label{eq:final-estimate}
\begin{aligned}
&\norm{\Psi_{j,m_j}^YF_{\bm n,\bm m}^{(j)}
((\Phi_{i,n_i}^X(x_i))_{i=1}^p)-\iota_j^Y(G_j(\bm x))}_{Z_j^Y}\\
\leq&
\norm{\Psi_{j,m_j}^YF_{\bm n,\bm m}^{(j)}
((\Phi_{i,n_i}^X(x_i))_{i=1}^p)-
\Psi_{j,m_j}^Y\mathcal H_j(\bm x)}_{Z_j^Y} +
\norm{\Psi_{j,m_j}^Y\Phi_{j,m_j}^Y(G_j(\bm x))-\iota_j^Y(G_j(\bm x))}_{Z_j^Y}\\
<&\frac{2\varepsilon}{3}<\varepsilon.
\end{aligned}
\end{equation}
Taking the supremum over $\bm x$ and the maximum over $j$ concludes the proof.
\end{proof}

\begin{corollary}
\label{cor:fixed-codeapprox}
Under the hypotheses of Theorem~\ref{thm:mimo-codeapprox}, fix $\bm m=(m_1,\ldots,m_q)$.
For every $\eta>0$, there exist $\bm n=(n_1,\ldots,n_p)$ and a continuous map $F_{\bm n,\bm m}$ of the form \eqref{eq:mimo-synthop} such that
\begin{equation}\label{eq:fixed-codeapprox}
\sup_{\bm x\in\bm K}\max_{1\le j\le q}
\norm{F_{\bm n,\bm m}^{(j)}
\bigl((\Phi_{i,n_i}^X(x_i))_{i=1}^p\bigr)
-\Phi_{j,m_j}^Y(G_j(\bm x))}_2<\eta.
\end{equation}
\end{corollary}
\begin{proof}
Fix $\bm m$ and repeat the argument in the proof of Theorem~\ref{thm:mimo-codeapprox}, using the prescribed $\eta$ as the code-space tolerance.
The approximation of $\mathcal H$ and the input-reconstruction estimate conclude the proof.
No output reconstruction or decoder estimate is needed.
\end{proof}

%% file: sec/proof2.tex
\subsubsection{Decoding estimates}
\label{sec:regapprox}
Fix $(\Gamma,f)\in\mathcal X$ where the interface $\Gamma$ is of class $C^2$, and write
\begin{equation*}
L=[f]_{\mathrm{Lip}(\Gamma)},\qquad
u=\sdf_\Gamma,\qquad v=\Ext_L^{\mathrm{mid}}[f].
\end{equation*}
Choose $0<\delta<\min\{\reach(\Gamma),\dist(\Gamma,\partial\Bx)\}$ and set $N_\delta=\{\xi\in\Bx:\dist(\xi,\Gamma)<\delta\}$.
Let $\nu$ denote the outward unit normal of $\Omega_\Gamma$.
For $w\in\{u,v\}$, let $c_r^w=\mathcal A_rw$ and let $\widehat c_r^w\in\R^{r^d}$ be a predicted code.
Set
\begin{equation*}
\Delta c_r^w=\widehat c_r^w-c_r^w,\qquad
\widetilde w=\mathcal S_r\widehat c_r^w,
\end{equation*}
and define $\widetilde\Gamma=\widetilde u^{-1}(0)$ and $\widetilde f=\widetilde v|_{\widetilde\Gamma}$.
For a nonempty set $A\subseteq\Bx$ and a scalar- or vector-valued function $g$, write $\norm{g}_{\infty,A}=\sup_{\xi\in A}\norm{g(\xi)}$.
Define
\begin{equation}\label{eq:errdefs}
\begin{aligned}
E_r^w(A)&=\norm{\widetilde w-w}_{\infty,A},&
\tau_r^w(A)&=\norm{\mathcal Q_rw-w}_{\infty,A},\\
\kappa_r(A)&=\sup_{\xi\in A}
\Big(\sum_{\bi\in I_r}\phi_{\bi}(\xi)^2\Big)^{1/2},&
\eta_r^w(A)&=\tau_r^w(A)+\kappa_r(A)\norm{\Delta c_r^w}_2.
\end{aligned}
\end{equation}
Since $\widetilde w-w=(\mathcal Q_rw-w)+\mathcal S_r\Delta c_r^w$, the Cauchy--Schwarz inequality yields:
\begin{equation}\label{eq:fieldbound}
E_r^w(A)\le\eta_r^w(A).
\end{equation}
If the reconstruction basis is $C^1$, define
\begin{equation}\label{eq:c1gain}
\begin{aligned}
\gamma_r^u&=\norm{\nabla(\widetilde u-u)}_{\infty,N_\delta},&
\tau_{r,1}^u&=\norm{\nabla(\mathcal Q_ru-u)}_{\infty,N_\delta},\\
\kappa_{r,1}&=\sup_{\xi\in N_\delta}
\Big(\sum_{\bi\in I_r}\norm{\nabla\phi_{\bi}(\xi)}^2\Big)^{1/2},&
\eta_{r,1}^u&=\tau_{r,1}^u+
\kappa_{r,1}\norm{\Delta c_r^u}_2,
\end{aligned}
\end{equation}
and similarly we have $\gamma_r^u\le\eta_{r,1}^u$.

\begin{lemma}[Decoding certificate]\label{lem:decode-certificate}
If $\eta_r^u(\Bx)<\delta$, then $\widetilde\Gamma$ is a nonempty compact subset of $N_\delta$, $\pi_\Gamma(\widetilde\Gamma)=\Gamma$, and
\begin{align}
d_H(\widetilde\Gamma,\Gamma)
&\le\eta_r^u(N_\delta),\label{eq:cert-hausdorff}\\
d_H^\times\bigl(\operatorname{gr}_{\widetilde\Gamma}\widetilde f,
\operatorname{gr}_\Gamma f\bigr)
&\le\max\bigl\{\eta_r^u(N_\delta),
\eta_r^v(N_\delta)+L\eta_r^u(N_\delta)\bigr\}.
\label{eq:certgraph}
\end{align}
If the basis is $C^1$ and $\eta_{r,1}^u<1$, then there is a unique $h\in C^1(\Gamma)$ such that
\begin{equation*}
\widetilde\Gamma=\{x+h(x)\nu(x):x\in\Gamma\},\qquad
\norm{h}_{L^\infty(\Gamma)}\le\eta_r^u(N_\delta)<\delta.
\end{equation*}
In addition, $\pi_\Gamma|_{\widetilde\Gamma}$ is a $C^1$ diffeomorphism and, with $q=(\pi_\Gamma|_{\widetilde\Gamma})^{-1}$,
\begin{equation*}
\nabla\widetilde u(q(x))\cdot\nu(x)\ge1-\eta_{r,1}^u>0,
\qquad x\in\Gamma.
\end{equation*}
Moreover, for $1\le p\le\infty$,
\begin{equation}\label{eq:transportLp}
\norm{\widetilde f\circ q-f}_{L^p(\Gamma)}
\le\abs{\Gamma}^{1/p}
\bigl(\eta_r^v(N_\delta)+L\eta_r^u(N_\delta)\bigr),
\end{equation}
where $\abs{\Gamma}$ is the surface measure and $\abs{\Gamma}^{1/\infty}=1$.
\end{lemma}

\begin{proof}
Denote $E_\Bx=E_r^u(\Bx)$ and $E_\delta=E_r^u(N_\delta)$.
By \eqref{eq:fieldbound}, for $\widetilde x\in\widetilde\Gamma$, we have
\begin{equation*}
\dist(\widetilde x,\Gamma)=\abs{u(\widetilde x)}
=\abs{u(\widetilde x)-\widetilde u(\widetilde x)}\le E_\Bx \le\eta_r^u(\Bx)<\delta,
\end{equation*}
so $\widetilde\Gamma\subset N_\delta$.
For each $x\in\Gamma$, if $E_\delta=0$ then $\widetilde u(x)=0$; otherwise, since
\begin{equation*}
\widetilde u(x+E_\delta\nu(x))\ge0,
\qquad
\widetilde u(x-E_\delta\nu(x))\le0,
\end{equation*}
the intermediate value theorem gives a zero on the normal fiber through $x$ at distance at most $E_\delta$.
This proves nonemptiness, $\pi_\Gamma(\widetilde\Gamma)=\Gamma$, and $d_H(\widetilde\Gamma,\Gamma)\le E_\delta$.
Compactness follows because $\widetilde\Gamma$ is closed in the compact box.
For $\widetilde x\in\widetilde\Gamma$, using the identity $v|_\Gamma=f$ and the bound $[v]_{\mathrm{Lip}(\Bx)}\le L$, we have
\begin{equation*}
\abs{\widetilde f(\widetilde x)-f(\pi_\Gamma\widetilde x)}
\le E_r^v(N_\delta)+L E_\delta.
\end{equation*}
Pairing each decoded point with its projection, and each $x\in\Gamma$ with the zero on its normal fiber constructed above, we conclude \eqref{eq:cert-hausdorff} and \eqref{eq:certgraph} by applying \eqref{eq:fieldbound}.

Now assume $\eta_{r,1}^u<1$.
For $x\in\Gamma$, taking $\varphi_x(t)=\widetilde u(x+t\nu(x))$, we have
\begin{equation*}
\varphi_x'(t)
=1+\nabla(\widetilde u-u)(x+t\nu(x))\cdot\nu(x)
\ge1-\gamma_r^u\ge1-\eta_{r,1}^u>0,
\end{equation*}
where we have used the fact that $u(x+t\nu(x))=t$ for $\abs{t}<\delta$.
Thus the zero on each normal fiber is unique; denote it by $x+h(x)\nu(x)$.
On the other hand, since every point of $\widetilde\Gamma$ lies in $N_\delta$, it admits a unique representation $\tilde x =x+s \nu (x)$ with $x\in\Gamma$.
The uniqueness of the zero on each normal fiber then yields $s=h(x)$, and hence $\widetilde\Gamma=\{x+h(x)\nu(x):x\in\Gamma\}$.
The implicit function theorem yields $h\in C^1(\Gamma)$.
Hence $ q(x)=x+h(x)\nu(x) $ is $C^1$.
By uniqueness of the zero on each normal fiber, $q$ is a bijection from $\Gamma$ onto $\widetilde\Gamma$, with inverse $ q^{-1}=\pi_\Gamma|_{\widetilde\Gamma}$.
Since the tubular projection $\pi_\Gamma$ is $C^1$, $q$ is a $C^1$ diffeomorphism.
The bounds on $h$ and $\nabla\widetilde u(q(x))\cdot\nu(x)$ follow from the construction.
Finally,
\begin{equation*}
\abs{\widetilde f(q(x))-f(x)}
\le E_r^v(N_\delta)+L\norm{q(x)-x}
\le\eta_r^v(N_\delta)+L\eta_r^u(N_\delta),
\end{equation*}
and taking the $L^p(\Gamma)$ norm we conclude \eqref{eq:transportLp}.
\end{proof}

\subsubsection{Cosine projection estimates}
The decoding lemma is basis independent. For the tensor cosine basis, the following proposition provides the estimates needed for the decoding lemma.
The $C^0$ bound in part~(i) follows from classical Jackson--Lebesgue theory.
The main contribution is the quantitative tube-local gradient estimate in part~(ii).
Its proof bounds contributions from distant singularities using the semiconcavity of the squared distance and the coordinate marginals of its positive-semidefinite Hessian defect measure.

\begin{proposition}[Cosine reconstruction estimates]\label{prop:approx}
Let $r\ge2$. For quantities defined in \eqref{eq:errdefs} and \eqref{eq:c1gain}, we have
\begin{enumerate}
\item[(i)] There is a constant $C_d\ge1$, depending only on $d$, such that for every Lipschitz function $w:\Bx\to\R$, 
\begin{equation}\label{eq:cosuniform}
\tau_r^w(\Bx)
\le C_d[w]_{\mathrm{Lip}(\Bx)}\frac{(1+\log r)^d}{r},
\qquad
\kappa_r(\Bx)=(2r-1)^{d/2}.
\end{equation}
\item[(ii)] Suppose that, for some $a>\delta$,
\begin{equation*}
\reach(\Gamma)\ge a,
\qquad
\dist(\Gamma,\partial\Bx)\ge a.
\end{equation*}
For fixed $d\ge2$, there is a constant $C_{a,\delta,d}$, independent of $\Gamma$, such that
\begin{equation}\label{eq:c1decay}
\begin{aligned}
\tau_{r,1}^u
&\le C_{a,\delta,d}\left\{
\frac{(1+\log(r+1))^d}{r}
+\frac{(1+\log(r+1))^{d-3/2}}{\sqrt r}
\right\},\\
\kappa_{r,1}
&\le\kappa_{r,1}^{\mathrm{box}}
:=\pi\sqrt{\frac{d\,r(r-1)}{3}}\,(2r-1)^{d/2}.
\end{aligned}
\end{equation}
\end{enumerate}
\end{proposition}

\begin{proof}

We first record a periodic approximation estimate used in both parts.
Set $\mathbb T_2=\R/2\mathbb Z, \ \mathbb T_2^d=(\mathbb T_2)^d, \ \varpi(t)=\dist(t,2\mathbb Z)$, write $\dist_{\mathbb T_2}([s],[t])=\dist(s-t,2\mathbb Z)$ and
\begin{equation*}
d_{\mathbb T_2^d}(x,y)
=\left(\sum_{i=1}^d\dist_{\mathbb T_2}(x_i,y_i)^2\right)^{1/2}.
\end{equation*}

For the $2$-periodic Dirichlet kernel, let
\begin{equation}
\begin{aligned}
D_n(t)&=1+2\sum_{k=1}^n\cos(\pi kt)
=\frac{\sin((n+\tfrac12)\pi t)}{\sin(\pi t/2)},\quad
\Lambda_n&=\frac12\int_{-1}^{1}\abs{D_n(t)}\,dt,
\end{aligned}
\end{equation}
where the quotient is defined by continuity at $2\mathbb Z$.
Indeed, the two expressions for $D_n$ imply $\abs{D_n(t)}\le\min\{2n+1,\abs{t}^{-1}\}$ for $0<\abs{t}\le1$.
Splitting the integral at $\abs{t}=(2n+1)^{-1}$ yields $\Lambda_n\le1+\log(2n+1)\le C(1+\log(n+1))$; see also \cite[Eq.~(12.1)]{Zygmund2002Trigonometric}.
Let $S_n^{\square}$ denote rectangular Fourier projection on $\mathbb T_2^d$, defined by the normalized integral:
\begin{equation*}
S_n^\square g(x)=2^{-d}\int_{\mathbb T_2^d}g(y)
\prod_{m=1}^dD_n(x_m-y_m)\,dy.
\end{equation*}
Tensorizing the one-dimensional Jackson estimate and applying Lebesgue's inequality
\cite[Theorem~(13.6) and Eq.~(13.25)]{Zygmund2002Trigonometric} yields, for every Lipschitz $F:\mathbb T_2^d\to\R$,
\begin{equation}\label{eq:periodic-jackson-lebesgue}
\norm{S_n^\square F-F}_{L^\infty(\mathbb T_2^d)}
\le C_d[F]_{\mathrm{Lip}(\mathbb T_2^d)}
\frac{(1+\log(n+1))^d}{n+1},\quad n\ge1.
\end{equation}
To prove part~(i), take the coordinatewise-even, $2$-periodic extension
\begin{equation*}
w^{\mathrm{ev}}(x)=w(\varpi(x_1),\ldots,\varpi(x_d)),
\end{equation*}
which satisfies $[w^{\mathrm{ev}}]_{\mathrm{Lip}(\mathbb T_2^d)}\le[w]_{\mathrm{Lip}(\Bx)}$.

Folding the integral onto $\Bx$ by coordinatewise evenness and using
\begin{equation*}
\frac{D_n(s-t)+D_n(s+t)}{2}
=1+2\sum_{k=1}^n\cos(\pi ks)\cos(\pi kt)
\end{equation*}
in each coordinate recovers the tensor cosine projection kernel.
Thus
\begin{equation*}
\mathcal Q_rw
=\bigl(S_{r-1}^{\square}w^{\mathrm{ev}}\bigr)|_\Bx, \quad \norm{\mathcal Q_r}_{C(\Bx)\to C(\Bx)}\le\Lambda_{r-1}^d.
\end{equation*}
Applying \eqref{eq:periodic-jackson-lebesgue} with $F=w^{\mathrm{ev}}$ and $n=r-1$, we obtain the asserted bound on $\tau_r^w(\Bx)$.
Moreover, we have
\begin{equation*}
\sum_{m=0}^{r-1}(\phi_m^{\mathrm{cos}}(x))^2
=1+2\sum_{m=1}^{r-1}\cos^2(\pi mx)\le2r-1.
\end{equation*}
Equality holds at $x=0$ and $x=1$.
And tensor factorization therefore yields $\kappa_r(\Bx)=(2r-1)^{d/2}$.
This completes part~(i).

We now prove part~(ii) under its reach and separation hypotheses.
Throughout this part, $C$ may change from line to line but depends only on $a$, $\delta$, and $d$.
The bound on $\kappa_{r,1}$ is easy to check.
In one dimension,
\begin{equation*}
\sum_{m=0}^{r-1}\abs{(\phi_m^{\mathrm{cos}})'(x)}^2
\le\frac{\pi^2}{3}r(r-1)(2r-1).
\end{equation*}
Consequently, tensor factorization yields that
\begin{equation*}
\sum_{\bi\in I_r}\norm{\nabla\phi_{\bi}^{\mathrm{cos}}(\xi)}^2
\le\frac{\pi^2}{3}d\,r(r-1)(2r-1)^d,
\end{equation*}
and hence $\kappa_{r,1} \le\kappa_{r,1}^{\mathrm{box}}$.
It remains to estimate $\tau_{r,1}^u$.

Put $n=r-1$ and $L_n=1+\log(n+2)$, and let $U$ be the coordinatewise-even, $2$-periodic extension of $u$ to $\mathbb T_2^d$.
We use the same notation for functions on the torus and their $2$-periodic lifts to $\R^d$.
Here $\partial_\ell$ denotes a weak derivative represented by a function, whereas $\mathbf{D}_\ell$ denotes a distributional derivative, which is a measure for a $BV$ function; $|\mathbf{D}_\ell p|$ denotes its total variation measure.
For smooth functions these derivatives agree with classical derivatives.
The notation $D_n$ continues to denote the Dirichlet kernel.
We write $\|\cdot\|_\infty$ for the $L^\infty$ essential supremum on the function's domain: $\mathbb T_2^d$ for fields, $\mathbb T_2$ for one-dimensional periodic functions, and $\R$ for scalar coefficient functions.
For continuous functions this agrees with the pointwise supremum used in $\|\cdot\|_{\infty,A}$.
The projection $S_n^\square$ defined above has frequency set $\{\mathbf k\in\mathbb Z^d:\norm{\mathbf k}_\infty\le n\}$.
Fix $j\in\{1,\ldots,d\}$.
Comparing Fourier coefficients, we obtain, distributionally,
\begin{equation}\label{eq:cos-fourier-gradient}
\partial_j\mathcal Q_ru=(S_n^\square\partial_jU)|_\Bx.
\end{equation}
Thus it remains to control $S_n^\square U_j-U_j$ uniformly on $N_\delta$, where $U_j=\partial_jU$.

We begin by introducing some notations and establishing a few basic properties.
Define the reflected periodic interface
\begin{equation*}
\Gamma_{\mathrm{per}}
=\{(\epsilon_i y_i+2m_i)_{i=1}^d:
y\in\Gamma,\ \epsilon\in\{-1,1\}^d,\ m\in\mathbb Z^d\}.
\end{equation*}
Let $F_{\mathrm{fold}}(x)=(\varpi(x_1),\ldots,\varpi(x_d))$.
For $x\in\R^d$ and $y\in\Bx$, we have
\begin{equation*}
\min_{\substack{\epsilon\in\{-1,1\}^d,\ m\in\mathbb Z^d}}
\norm{x-(\epsilon_i y_i+2m_i)_{i=1}^d}^2
=\norm{F_{\mathrm{fold}}(x)-y}^2.
\end{equation*}
Taking the infimum over $y\in\Gamma$, we therefore obtain
\begin{equation*}
W:=U^2=\dist(\,\cdot\,,\Gamma_{\mathrm{per}})^2.
\end{equation*}

For every nonempty closed set $A\subset\R^d$, the function $x\longmapsto \dist(x,A)^2-\norm{x}^2 =\inf_{y\in A}\bigl(\norm{y}^2-2x\cdot y\bigr)$ is concave.
Thus
\begin{equation}\label{eq:hessian-defect}
M:=2I_d\,\mathcal L^d-\mathbf{D}^2W
\end{equation}
is a positive-semidefinite, $2$-periodic matrix-valued Radon measure.
Regard $M$ as a measure on one period, namely on $\mathbb T_2^d$.
For $q\in\{1,\ldots,d\}$, let $y_{\widehat q}$ denote the coordinates other than $y_q$.
For every nonnegative Borel function $h:\mathbb T_2^{d-1}\to[0,\infty]$, we have
\begin{equation}\label{eq:hessian-marginal}
\int_{\mathbb T_2^d}h(y_{\widehat q})\,dM_{qq}(y)
=4\int_{\mathbb T_2^{d-1}}h(z)\,dz.
\end{equation}
In fact, for smooth $h$, the lift $h(y_{\widehat q})$ is independent of $y_q$, so its second derivative in that coordinate vanishes.
Testing \eqref{eq:hessian-defect} against this lift leaves only $2\mathcal L^d$; integration over the $q$th coordinate contributes a further factor $2$.
Equality on smooth tests identifies the corresponding measures on $\mathbb T_2^{d-1}$, yielding \eqref{eq:hessian-marginal} for all nonnegative Borel $h$.

We will also use the following Cauchy--Schwarz inequality for the entries of $M$
\cite[Proposition III.24]{Moszynski2022Spectral}:
for Borel $\Xi\ge0$,
\begin{equation}\label{eq:matrix-measure-cs}
\int \Xi\, d| M_{ij}|
\le\left(\int \Xi\,dM_{ii}\right)^{1/2}
     \left(\int \Xi\,dM_{jj}\right)^{1/2}.
\end{equation}

We then construct the following near--far decomposition.
Choose $\delta<\delta_1<\delta_2<a$, with $\delta_1$ and $\delta_2$ depending only on $a$ and $\delta$, and choose an even $b\in C^\infty(\R)$ depending only on $a$ and $\delta$, such that $b=0$ on $[-\delta_1,\delta_1]$ and $b=1$ outside $(-\delta_2,\delta_2)$.
Set
\begin{equation*}
U_j^{\mathrm{near}}=(1-b(U))U_j,
\qquad
U_j^{\mathrm{far}}=b(U)U_j.
\end{equation*}
Distinct reflected copies of $\Gamma$ are separated by at least $2a$ because $\Gamma$ stays at distance at least $a$ from every face of $\Bx$, and each copy has reach at least $a$.
Hence $\{|U|<\delta_2\}$ is the disjoint union of their $\delta_2$-reach tubes.
On each tube, $U$ is a signed distance function.
Applying the distance-function formula on both sides of each copy
\cite[Lemmas~14.16--14.17]{Gilbarg2001Elliptic}
and using the curvature bound $|\kappa_i|\le a^{-1}$ implied by $\reach(\Gamma)\ge a$, we obtain
\begin{equation*}
\norm{\mathbf{D}^2U(x)}_{\mathrm{op}}
\le\frac{1}{a-|U(x)|}\le\frac{1}{a-\delta_2}
\qquad (|U(x)|<\delta_2).
\end{equation*}
Since $1-b(U)$ vanishes outside these tubes, $U_j^{\mathrm{near}}$ has a periodic Lipschitz representative with a uniform Lipschitz bound.
Applying \eqref{eq:periodic-jackson-lebesgue} with $F=U_j^{\mathrm{near}}$ yields
\begin{equation}\label{eq:local-gradient-part}
\norm{S_n^\square U_j^{\mathrm{near}}-U_j^{\mathrm{near}}}
_{L^\infty(\mathbb T_2^d)}
\le C\frac{L_n^d}{n+1}.
\end{equation}
For the far piece, define $\alpha_b(t)=b(t)/(2t)$ for $t\ne0$ and $\alpha_b(0)=0$.
Since $b$ vanishes near zero, $\alpha_b$ is smooth and both $\alpha_b$ and $\alpha_b'$ are uniformly bounded.
Moreover, we have $U_j^{\mathrm{far}}=\alpha_b(U)\partial_jW$, since $\partial_jW=2UU_j$ almost everywhere.
Since $U$ is the even periodic extension of a signed distance function, $\|\nabla U\|_\infty\le1$ and $\|U\|_\infty\le\sqrt d$.
Together with the bounds on $b$, $\alpha_b$, and $\alpha_b'$, we deduce the uniform bounds
\begin{equation}\label{eq:far-field-uniform-bounds}
\begin{aligned}
\|\alpha_b\|_\infty+\|\alpha_b'\|_\infty&\le C,\quad
\|\nabla U\|_\infty+\|\partial_jW\|_\infty
+\|U_j^{\mathrm{far}}\|_\infty&\le C,
\end{aligned}
\end{equation}
where field norms are taken over $\mathbb T_2^d$ and coefficient norms over $\R$.
Since $M_{ii}=2\mathcal L^d-\mathbf{D}_{ii}W$, testing against the constant periodic function $1$, we obtain
\begin{equation*}
M_{ii}(\mathbb T_2^d)
=2\mathcal L^d(\mathbb T_2^d)-\langle W,\partial_{ii}1\rangle
=2^{d+1}.
\end{equation*}
For every Borel set $E$, the matrix $M(E)$ is positive semidefinite, so
\begin{equation*}
|M_{\ell j}(E)|
\le\sqrt{M_{\ell\ell}(E)M_{jj}(E)}
\le\tfrac12\bigl(M_{\ell\ell}(E)+M_{jj}(E)\bigr).
\end{equation*}
Since $\mathbf{D}_\ell(\partial_jW)=2\delta_{\ell j}\mathcal L^d-M_{\ell j}$, where $\delta_{\ell j}$ is the Kronecker delta, we have
\begin{equation*}
\begin{aligned}
|\mathbf{D}_\ell(\partial_jW)|(\mathbb T_2^d)
\le 2\delta_{\ell j}\mathcal L^d(\mathbb T_2^d)
+|M_{\ell j}|(\mathbb T_2^d)\le 2^{d+1}(1+\delta_{\ell j})<\infty.
\end{aligned}
\end{equation*}
Together with the boundedness of $\partial_jW=2UU_j$, this proves $\partial_jW\in BV(\mathbb T_2^d)$.
Moreover, $U\in W^{1,\infty}$ and $\alpha_b\in C_b^1$ imply $\alpha_b(U)\in W^{1,\infty}$.
The BV chain and product rules~\cite[Theorem~3.96 and Example~3.97]{Ambrosio2000Functions} therefore yield
\begin{equation}\label{eq:remote-bv-derivative}
\mathbf{D}_\ell U_j^{\mathrm{far}}
=\bigl(2\alpha_b(U)\delta_{\ell j}
+\alpha_b'(U)\partial_\ell U\,\partial_jW\bigr)\mathcal L^d
-\alpha_b(U)M_{\ell j}.
\end{equation}
In particular, $U_j^{\mathrm{far}}\in BV(\mathbb T_2^d)\cap L^\infty(\mathbb T_2^d)$.

We next present the elementary kernel estimate used below.
Fix $x\in\mathbb T_2^d$ and $\ell\in\{1,\ldots,d\}$.
Suppose that $p\in BV(\mathbb T_2^d)\cap L^\infty(\mathbb T_2^d)$ and its support is contained in
\begin{equation*}
\{y:\dist_{\mathbb T_2}(x_\ell,y_\ell)\ge\varsigma\}
\end{equation*}
for some $\varsigma>0$.
Then there exists a constant $C_\varsigma$ depending only on $d$ and $\varsigma$ such that
\begin{equation}\label{eq:off-support-bv}
\begin{aligned}
&\left|\int_{\mathbb T_2^d}p(y)
\prod_{m=1}^dD_n(x_m-y_m)\,dy\right|\\
&\le\frac{C_\varsigma}{n+1}\left\{
\int_{\mathbb T_2^d}\prod_{m\ne\ell}|D_n(x_m-y_m)|
\,d|\mathbf{D}_\ell p|(y)
+\norm{p}_{L^\infty(\mathbb T_2^d)}L_n^{d-1}\right\}.
\end{aligned}
\end{equation}
The proof is given below.
Let $K(y)=\prod_{m\ne\ell}D_n(x_m-y_m)$.
For $t\in\mathbb T_2$, put $s(t)=\dist_{\mathbb T_2}(x_\ell,t)$.
Write $A_n=(n+\tfrac12)\pi$ and $z=x_\ell-t$.
For $s(t)>0$, define
\begin{equation*}
R(t)=\frac{\cos(A_nz)}{A_n\sin(\pi z/2)},\qquad
E(t)=-\frac{\pi}{2A_n}
\frac{\cos(A_nz)\cos(\pi z/2)}{\sin^2(\pi z/2)}.
\end{equation*}
Differentiating $R$, using $\partial_tz=-1$, we obtain $R'(t)=D_n(x_\ell-t)-E(t)$.
If $t=y_\ell$ for some $y\in\operatorname{supp}p$, the separation assumption yields $\varsigma\le s(t)\le1$, hence
\begin{equation*}
|\sin(\pi z/2)|=\sin(\pi s(t)/2)
\ge\sin(\pi\varsigma/2)>0.
\end{equation*}
Consequently, for these values of $t$,
\begin{equation*}
D_n(x_\ell-t)=R'(t)+E(t),\qquad
|R(t)|+|E(t)|\le\frac{C_\varsigma}{n+1}.
\end{equation*}
Thus integration of the oscillatory numerator produces the factor $(n+1)^{-1}$.
However, $R$ is singular at $s(t)=0$ and cannot yet be used as a smooth test function on the whole torus.
Choose a smooth $2$-periodic cutoff $\zeta:\mathbb T_2\to[0,1]$ such that
\begin{equation*}
\zeta(t)=0\quad\text{if }s(t)\le\varsigma/3,
\qquad
\zeta(t)=1\quad\text{if }s(t)\ge2\varsigma/3.
\end{equation*}
Define $\widehat R(t)=\zeta(t)R(t)$ for $s(t)>0$, extending it by zero at $s(t)=0$, and define $ \widehat E(t)=D_n(x_\ell-t)-\widehat R'(t)$.
For $y\in\operatorname{supp}p$, we have $s(y_\ell)\ge\varsigma$, so $\zeta$ is identically one near $t=y_\ell$.
Thus $\widehat R'(y_\ell)=R'(y_\ell)$ and $\widehat E(y_\ell)=E(y_\ell)$.
In particular,
\begin{equation*}
\|\widehat R\|_\infty\le\frac{C_\varsigma}{n+1},\qquad
|\widehat E(y_\ell)|\le\frac{C_\varsigma}{n+1}
\quad(y\in\operatorname{supp}p).
\end{equation*}
For $p\in BV$, integration by parts means $\int p\,\partial_\ell\psi\,dy=-\int\psi\,d\mathbf{D}_\ell p$ for smooth periodic $\psi$.
Taking $\psi(y)=\widehat R(y_\ell)K(y)$ and using $\partial_\ell K=0$, we have
\begin{equation*}
\begin{aligned}
\int_{\mathbb T_2^d}p(y)D_n(x_\ell-y_\ell)K(y)\,dy
&=-\int_{\mathbb T_2^d}\widehat R(y_\ell)K(y)\,d\mathbf{D}_\ell p(y) +\int_{\mathbb T_2^d}p(y)\widehat E(y_\ell)K(y)\,dy.
\end{aligned}
\end{equation*}
Using the bounds on $\widehat R$ and $\widehat E$, together with the total variation inequality for the measure integral, we obtain that
\begin{equation*}
\left|\int_{\mathbb T_2^d}p(y)
\prod_{m=1}^dD_n(x_m-y_m)\,dy\right|\leq \frac{C_\varsigma}{n+1}
\left(\int_{\mathbb T_2^d}|K|\,d|\mathbf{D}_\ell p|
+\|p\|_\infty\int_{\mathbb T_2^d}|K|\,dy\right).
\end{equation*}
Applying the estimate that $ \int_{\mathbb T_2^d}|K(y)|\,dy =2\|D_n\|_{L^1(\mathbb T_2)}^{d-1}\le C L_n^{d-1} $, we conclude \eqref{eq:off-support-bv}.

We now apply this estimate to $U_j^{\mathrm{far}}$, absorbing the fixed normalizing factor $2^{-d}$ into $C$.
Regard $\overline{N_\delta}\subset\Bx$ as a subset of the torus.
If $x\in\overline{N_\delta}$ and $y\in\operatorname{supp}U_j^{\mathrm{far}}$, then $|U(x)|\le\delta$ and $|U(y)|\ge\delta_1$.
Since $|U|=\dist(\,\cdot\,,\Gamma_{\mathrm{per}})$ is $1$-Lipschitz, we have $d_{\mathbb T_2^d}(x,y)\ge\delta_1-\delta$. Thus some coordinate distance is at least $(\delta_1-\delta)/\sqrt d$.
The associated coordinate-separation sets cover $\operatorname{supp}U_j^{\mathrm{far}}$.
For each fixed $x$, let $\sigma=(\delta_1-\delta)/\sqrt d$.
The coordinate-distance vector has Euclidean norm at least $\delta_1-\delta$, so its largest component is at least $\sigma$.
Choose a smooth periodic cutoff $\chi:\mathbb T_2\to[0,1]$ that is zero when $\dist_{\mathbb T_2}(t,0)\le\sigma/2$ and one when $\dist_{\mathbb T_2}(t,0)\ge\sigma$, with $C^1$ norm depending only on $\sigma$.
Set
\begin{equation*}
\chi_\ell^x(y)=\chi(y_\ell-x_\ell),\qquad
H_x(y)=\sum_{m=1}^d\chi_m^x(y).
\end{equation*}
Then $H_x\ge1$ on $\operatorname{supp}U_j^{\mathrm{far}}$.
To normalize smoothly even where $H_x=0$, choose $\rho\in C^\infty([0,\infty);[0,1])$ with $\rho=0$ on $[0,1/4]$ and $\rho=1$ on $[1/2,\infty)$, and define
\begin{equation*}
\theta_\ell^x(y)=\rho(H_x(y)){\chi_\ell^x(y)}/{H_x(y)},
\end{equation*}
with value zero when $H_x=0$.
These functions are smooth, lie in $[0,1]$, and sum to one on the far-field support.
Their $C^1$ bounds depend only on $\sigma$ and $d$, since division occurs only where $H_x>1/4$.
In particular,
\begin{equation}\label{eq:partition-uniform-bounds}
\|\theta_\ell^x\|_\infty\le1,\qquad
\|\nabla\theta_\ell^x\|_\infty\le C,
\end{equation}
uniformly in $x$ and $\ell$.
Moreover,
\begin{equation*}
\operatorname{supp}(\theta_\ell^xU_j^{\mathrm{far}})
\subseteq\{y:\dist_{\mathbb T_2}(x_\ell,y_\ell)\ge\sigma/2\}.
\end{equation*}
Multiplication by the smooth cutoff preserves $BV\cap L^\infty$, so \eqref{eq:off-support-bv} applies to each piece $p=\theta_\ell^xU_j^{\mathrm{far}}$ with the same separation parameter $\varsigma=\sigma/2$.

Fix $x\in\overline{N_\delta}$ and set $p=\theta_\ell^xU_j^{\mathrm{far}}$ for each $\ell$.
First, \eqref{eq:far-field-uniform-bounds} and \eqref{eq:partition-uniform-bounds} imply
\begin{equation*}
\|p\|_\infty
\le\|\theta_\ell^x\|_\infty\|U_j^{\mathrm{far}}\|_\infty
\le C.
\end{equation*}
Next, by the product rule and \eqref{eq:remote-bv-derivative}, we have
\begin{equation*}
\begin{aligned}
\mathbf{D}_\ell p
&=(\partial_\ell\theta_\ell^x)U_j^{\mathrm{far}}\,\mathcal L^d
  +\theta_\ell^x\mathbf{D}_\ell U_j^{\mathrm{far}}\\
&=\Bigl[(\partial_\ell\theta_\ell^x)U_j^{\mathrm{far}}
  +\theta_\ell^x\bigl(2\alpha_b(U)\delta_{\ell j}
  +\alpha_b'(U)\partial_\ell U\,\partial_jW\bigr)\Bigr]\mathcal L^d -\theta_\ell^x\alpha_b(U)M_{\ell j}.
\end{aligned}
\end{equation*}
Using \eqref{eq:far-field-uniform-bounds} and \eqref{eq:partition-uniform-bounds} again, all scalar factors above are uniformly bounded.
Taking total variation, we have
\begin{equation*}
|\mathbf{D}_\ell p|\le C\mathcal L^d+C|M_{\ell j}|.
\end{equation*}
With $\Xi(y)=\prod_{m\ne\ell}|D_n(x_m-y_m)|$, tensor factorization gives $\int\Xi\,dy=2\|D_n\|_{L^1(\mathbb T_2)}^{d-1}\le CL_n^{d-1}$.
Substituting these bounds into \eqref{eq:off-support-bv} yields
\begin{equation}
\label{eq:off-support-bv-2}
\left|\int_{\mathbb T_2^d}p(y)
\prod_{m=1}^dD_n(x_m-y_m)\,dy\right|
\le\frac{C}{n+1}
\left(L_n^{d-1}+\int_{\mathbb T_2^d}\Xi\,d|M_{\ell j}|\right).
\end{equation}
Thus it remains to bound the last integral; all cutoff weights have already been absorbed into $C$.

Since $\Xi$ is independent of $y_\ell$, the marginal identity \eqref{eq:hessian-marginal} yields
\begin{equation*}
\int\Xi\,dM_{\ell\ell}
=4\|D_n\|_{L^1(\mathbb T_2)}^{d-1}
\le CL_n^{d-1}.
\end{equation*}
For $\ell=j$, this is the required bound, because $M_{jj}$ is nonnegative.
For $\ell\ne j$, first bound the $j$th kernel in $\Xi$ by $\|D_n\|_\infty=2n+1$.
The remaining product is independent of $y_j$, so using the same marginal identity we derive that
\begin{equation*}
\begin{aligned}
\int\Xi\,dM_{jj}
&\le(2n+1)\int\prod_{m\ne\ell,j}|D_n(x_m-y_m)|\,dM_{jj}(y)\\
&\le C(n+1)\|D_n\|_{L^1(\mathbb T_2)}^{d-2}
\le C(n+1)L_n^{d-2}.
\end{aligned}
\end{equation*}
There are $d-2$ kernels here; integration over the unused coordinate $y_\ell$ contributes only a factor $2$.
Applying \eqref{eq:matrix-measure-cs} now yields
\begin{equation*}
\begin{aligned}
\int\Xi\,d|M_{\ell j}|
&\le\left(\int\Xi\,dM_{\ell\ell}\right)^{1/2}
     \left(\int\Xi\,dM_{jj}\right)^{1/2}\le C\sqrt{n+1}\,L_n^{d-3/2}.
\end{aligned}
\end{equation*}

For each fixed $x$, the partition satisfies $\sum_{\ell=1}^d\theta_\ell^xU_j^{\mathrm{far}}=U_j^{\mathrm{far}}$.
Applying \eqref{eq:off-support-bv-2} with the bounds established above, we obtain
\begin{equation*}
\begin{aligned}
|S_n^\square U_j^{\mathrm{far}}(x)|
&=2^{-d}\left|\sum_{\ell=1}^d\int_{\mathbb T_2^d}
\theta_\ell^x(y)U_j^{\mathrm{far}}(y)
\prod_{m=1}^dD_n(x_m-y_m)\,dy\right|\\
&\le\frac{C}{n+1}
\left(L_n^{d-1}+\sqrt{n+1}\,L_n^{d-3/2}\right).
\end{aligned}
\end{equation*}
Since $C$ is independent of $x$, taking the supremum, we have
\begin{equation}\label{eq:remote-gradient-part}
\sup_{x\in\overline{N_\delta}}|S_n^\square U_j^{\mathrm{far}}(x)|
\le C\left\{
\frac{L_n^{d-1}}{n+1}
+\frac{L_n^{d-3/2}}{\sqrt{n+1}}
\right\}.
\end{equation}

On $\overline{N_\delta}$, $U_j^{\mathrm{far}}=0$ and $U_j^{\mathrm{near}}=U_j$.
Hence \eqref{eq:cos-fourier-gradient} yields
\begin{equation*}
\partial_j\mathcal Q_ru-\partial_ju
=S_n^\square U_j^{\mathrm{near}}-U_j^{\mathrm{near}}
+S_n^\square U_j^{\mathrm{far}}.
\end{equation*}
Combining \eqref{eq:local-gradient-part} and \eqref{eq:remote-gradient-part}, bounding the Euclidean norm of the gradient error by the sum of its componentwise absolute values, and absorbing $L_n^{d-1}/(n+1)$ into $L_n^d/(n+1)$ proves
\begin{equation*}
\tau_{r,1}^u
\le C_{a,\delta,d}\left\{
\frac{(1+\log(r+1))^d}{r}
+\frac{(1+\log(r+1))^{d-3/2}}{\sqrt r}
\right\}.
\end{equation*}
Here we used $n+1=r$ and $L_n=1+\log(r+1)$.
All constants depend only on $a$, $\delta$, and $d$.
Together with the bound for $\kappa_{r,1}$, we conclude \eqref{eq:c1decay} and complete part~(ii).
\end{proof}

\subsubsection{Completion of the proof}
\begin{proof}
Use $\varepsilon$ as the tube radius in the preceding estimates.
Put $\varepsilon_1=\varepsilon/(1+\Lambda)$.
For any $y=(\Gamma,f)\in\mathcal G(K)$ set $u=\sdf_\Gamma$ and $v=\Ext_{[f]_{\mathrm{Lip}(\Gamma)}}^{\mathrm{mid}}[f]$; hence $[u]_{\mathrm{Lip}}\le1$ and $[v]_{\mathrm{Lip}}\le\Lambda$.
By (H2) and Proposition~\ref{prop:approx}, we can choose $m=r_{\mathrm{out}}\ge2$, uniformly in $y$, such that
\begin{equation}\label{eq:opchoice-m}
\tau_m^u(\Bx)\le\varepsilon_1/2,\qquad
\tau_m^v(\Bx)\le\varepsilon_1/2,\qquad \tau_{m,1}^u\le1/2.
\end{equation}
Set $\rho=\min\{\varepsilon_1/(2\kappa_m(\Bx)), 1/(4\kappa_{m,1}^{\mathrm{box}})\}>0$.
Proposition~\ref{prop:approx}(i) verifies the hypotheses of Corollary~\ref{cor:fixed-codeapprox}.
Thus there exist an integer $n=r_{\mathrm{in}}\ge2$ and a continuous function $F:\R^{2n^d}\to\R^{2m^d}$ such that
\begin{equation}\label{eq:opcode}
\sup_{x\in K}\norm{F(\mathcal E_nx)-\mathcal E_m(\mathcal G(x))}_2<\rho.
\end{equation}
For any $x\in K$, write $\mathcal G(x)=(\Gamma,f)$, let $(U,V)$ be synthesized from $F(\mathcal E_n x)$, and set $\Sigma=U^{-1}(0)$ and $g=V|_{\Sigma}$.
Combining estimates \eqref{eq:opchoice-m} and \eqref{eq:opcode} we have
\begin{equation*}
\eta_m^u(\Bx),\eta_m^v(\Bx)<\varepsilon_1\le\varepsilon,
\qquad \eta_{m,1}^u<3/4.
\end{equation*}
Applying Lemma~\ref{lem:decode-certificate}, we conclude the claimed unique normal graph and projection diffeomorphism and, for $q=(\pi_\Gamma|_\Sigma)^{-1}$,
\begin{equation*}
d_H(\Sigma,\Gamma)<\varepsilon_1\le\varepsilon,\quad
d_H^\times(\operatorname{gr}_\Sigma g,\operatorname{gr}_\Gamma f)<\varepsilon,
\quad \norm{g\circ q-f}_{L^p(\Gamma)}<\abs{\Gamma}^{1/p}\varepsilon.
\end{equation*}

We now show that $(\Sigma,g)\in\mathcal X$.
Let $T(y,t)=y+t\nu_\Gamma(y)$ and $H=\norm h_\infty<\varepsilon$; choose $\varepsilon<b<a$ and $\chi\in C_c^1((-b,b))$ with $\chi(0)=1$, $H\norm{\chi'}_\infty<1$.
Then the map
\begin{equation*}
\Phi(x)=
\begin{cases}
T\bigl(y,t+h(y)\chi(t)\bigr),
&x=T(y,t)\in N_b,\\[1mm]
x,&x\notin N_b.
\end{cases}
\end{equation*}
defines a $C^1$ diffeomorphism of $\Bx$ and $\Phi(\Gamma)=\Sigma$.
The estimate $\norm{U-u}_\infty<\varepsilon$ implies that $U$ and $u$ share the same signs off the $\varepsilon$-tube.
Since $\partial_tU(T(y,t))\ge1-\eta_{m,1}^u>0$ for $|t|<\varepsilon$, we have
\begin{equation*}
U(T(y,s))<0\quad\Longleftrightarrow\quad s<h(y).
\end{equation*}
Because $t+h(y)\chi(t)$ is strictly increasing with respect to $t$, we obtain that
\begin{equation*}
t<0
\quad\Longleftrightarrow\quad
t+h(y)\chi(t)<h(y)
\quad\Longleftrightarrow\quad
U(\Phi(T(y,t)))<0.
\end{equation*}
Since $t<0$ describes $\Omega_\Gamma$ in $N_b$ and $\Phi$ is the identity off $N_b$, we conclude that
\begin{equation*}
\{U<0\}=\Phi(\Omega_\Gamma).
\end{equation*}
Define $\Omega_\Sigma:=\Phi(\Omega_\Gamma)$.
It is nonempty, open, and compactly contained in $\operatorname{int}\Bx$, and
\begin{equation*}
\partial\Omega_\Sigma
=\Phi(\partial\Omega_\Gamma)
=\Phi(\Gamma)=\Sigma.
\end{equation*}
In addition, $V\in V_m\subset C^1(\Bx)$ implies that $g=V|_\Sigma$ is Lipschitz.
Thus $(\Sigma,g)\in\mathcal X$.

Finally, set
\begin{equation*}
e_*=\max_{x\in K}\norm{F(\mathcal E_nx)-\mathcal E_m(\mathcal G(x))}_2,
\qquad \rho_0=\rho-e_*>0.
\end{equation*}
If a continuous $F'$ satisfies the stated strict $\rho_0$ bound, the triangle inequality yields \eqref{eq:opcode} with $F'$, so the whole argument applies.
Since $\mathcal E_n(K)$ is compact, such an $F'$ may be chosen as a neural network \cite{Chen1995Universal,Jin2022MIONet}.
\end{proof}

%% file: sec/numerical_example.tex
We use three experiments to test distinct aspects of the IAE: cross-space prediction on variable Poisson domains, interface prediction through Hele--Shaw mergers, and joint interface--surfactant prediction in Stokes flow.
Except for the standalone GINO comparison, all learned maps are product MIONets~\cite{Jin2022MIONet}.
Unless stated otherwise, errors are evaluated on the full held-out set and reported as the mean and standard deviation over five independent training runs.
Detailed data generation, discretization, network, optimization, and postprocessing protocols are documented in Appendices~\ref{app:numerical-shared}--\ref{app:numerical-stokes}. 

\subsection{Variable-domain Poisson problems}
\label{sec:ex1}
We first demonstrate, with cross-space geometric operators as a case study, that the IAE can effectively encode information and outperforms existing approaches in terms of accuracy. Consider the Poisson equation:
\begin{equation}
 \begin{aligned}
 -\nabla\!\cdot\!\bigl(k(\xi)\nabla u(\xi)\bigr)&=s(\xi)
 &&\text{in }\Omega, \qquad k(\xi)>0\\
 u(\xi)&=g(\xi)&&\text{on }\partial\Omega.
 \end{aligned}
 \label{eq:poisson}
\end{equation}
The target operator is
\[
 \mathcal G_{\mathrm P}:(\Omega,k,s,g)\longmapsto u,
\]
where $u$ solves Eq.~\eqref{eq:poisson}.
The domains are either smooth radial domains or radial domains with a localized boundary notch:
\begin{equation}
 \Omega=\left\{\xi_0+s(\cos\theta,\sin\theta):
 0\le s<\varrho(\theta),\ 0\le\theta<2\pi\right\}.
 \label{eq:poissondomain}
\end{equation}
Here $\varrho$ is the unnotched radius or its modification by a localized
Gaussian notch, as specified in Appendix~\ref{app:numerical-poisson}.
The dataset contains $40000$ training samples and $1000$ test samples.

Writing $\Gamma=\partial\Omega$, the 2D-MFE and 1D-MFE encode geometry by area and boundary moments, respectively,
\begin{equation}
 \int_\Omega\phi_i^{\mathrm{Leg}}(\xi_1)
 \phi_j^{\mathrm{Leg}}(\xi_2)\,d\xi,
 \qquad
 \int_\Gamma\phi_i^{\mathrm{Leg}}(\xi_1)
 \phi_j^{\mathrm{Leg}}(\xi_2)\,ds,
 \label{eq:mfedgeom}
\end{equation}
where $\phi_m^{\mathrm{Leg}}(x)=\sqrt{2m+1}\,P_m(2x-1)$ and $P_m(1)=1$.
Both MFE variants encode the boundary condition by
\begin{equation}
 (z_g)_{ij}=\int_\Gamma g(\xi)
 \phi_i^{\mathrm{Leg}}(\xi_1)
 \phi_j^{\mathrm{Leg}}(\xi_2)\,ds.
 \label{eq:mfeboundary}
\end{equation}
The IAE uses either the cosine or Legendre basis functions to encode the geometry and boundary conditions.
All models use the same domain moments for $h\in\{k,s\}$,
\begin{equation}
 (z_h)_{ij}=\int_\Omega h(\xi)
 \phi_i^{\mathrm{Leg}}(\xi_1)
 \phi_j^{\mathrm{Leg}}(\xi_2)\,d\xi,
 \qquad 0\le i,j<r.
 \label{eq:mfefield}
\end{equation}
Each input code block has order $r=12$ and dimension $144$.
For $h=s$, we denote the source code by $z_{\mathrm{src}}$.
These four encoding configurations use the same four-branch value-head MIONet,
\begin{equation}
\begin{gathered} \begin{aligned}
 F_\theta^{(1)}:
 (z_{\mathrm{geom}},z_k,z_{\mathrm{src}},z_g)&\longmapsto\widehat u,\\
 \widehat u(\xi)
 &=\inner{
 \mathcal B_{\mathrm{geom}}(z_{\mathrm{geom}})\odot
 \mathcal B_k(z_k)\odot
 \mathcal B_{s}(z_{\mathrm{src}})\odot
 \mathcal B_g(z_g)}{\mathcal T(\xi)}+\beta.
 \end{aligned}
\end{gathered}
\label{eq:ex1net}
\end{equation}

The comparison uses the raw code (\emph{none}), coordinatewise standardization (\emph{per-dim}), and low-frequency rescaling (\emph{low-freq}).
For training codes $c^{(n)}\in\R^q$, $n=1,\ldots,N_{\mathrm{tr}}$, the two normalized forms are
\begin{equation}
\begin{aligned}
&\mu_j
=\frac1{N_{\mathrm{tr}}}\sum_{n=1}^{N_{\mathrm{tr}}}c_j^{(n)},\quad
\sigma_j^2:=\frac1{N_{\mathrm{tr}}}
\sum_{n=1}^{N_{\mathrm{tr}}}(c_j^{(n)}-\mu_j)^2,\quad
\sigma_{\mathrm{glob}}^2:=\frac1q\sum_{j=1}^q\sigma_j^2,\\
&\widetilde c_j^{\mathrm{per}}= \frac{c_j-\mu_j}{\sigma_j},
\quad 
\widetilde c_{\bi}^{\mathrm{lf}}
=\frac{\omega^{\mathrm{in}}_{\bi}}{\langle\omega^{\mathrm{in}}\rangle} \frac{c_{\bi}-\mu_{\bi}}{\sigma_{\mathrm{glob}}} ,\quad 
\omega^{\mathrm{in}}_{\bi}
=\frac{1}{1+\alpha_{\mathrm{in}}\norm{\bi}_2^2}.
\end{aligned}
\label{eq:numnorm}
\end{equation}
Here $\langle\omega^{\mathrm{in}}\rangle$ denotes the average over $\bi\in I_r$ and we take $\alpha_{\mathrm{in}}=0.5$.
We report the all-node relative discrete $\ell^2$ error, averaged over the $S$ test samples,
\begin{equation}
 \relL(\widehat u,u)=\frac{\norm{\widehat u-u}_2}{\norm{u}_2},
 \qquad
 \overline{\relL}=\frac1S\sum_{s=1}^S
 \relL(\widehat u_s,u_s).
 \label{eq:relL2}
\end{equation}

\begin{table}[tbp]
\centering
\caption{Relative prediction error for variable-domain Poisson problems. The results are obtained by taking the mean of 5 independent experiments, and the errors are recorded in the form of mean $\pm$ standard deviation. For each model, boldface marks the descriptive minimum among its three displayed normalization results.}
\label{tab:ex1}
\small
\setlength{\tabcolsep}{6pt}
\begin{tabular}{llccc}
\toprule
model / encoder & normalization & overall & notched & smooth \\
\midrule
2D-MFE & none & $0.0899{\pm}0.0037$ & $0.1144{\pm}0.0047$ & $0.0664{\pm}0.0030$ \\
 & per-dim & $\mathbf{0.0858{\pm}0.0013}$ & $0.1135{\pm}0.0016$ & $0.0593{\pm}0.0013$ \\
 & low-freq & $0.0880{\pm}0.0014$ & $0.1173{\pm}0.0019$ & $0.0600{\pm}0.0009$ \\
\addlinespace[2pt]
1D-MFE & none & $0.0921{\pm}0.0015$ & $0.1102{\pm}0.0016$ & $0.0748{\pm}0.0015$ \\
 & per-dim & $\mathbf{0.0653{\pm}0.0011}$ & $0.0826{\pm}0.0013$ & $0.0487{\pm}0.0010$ \\
 & low-freq & $0.0664{\pm}0.0006$ & $0.0864{\pm}0.0009$ & $0.0472{\pm}0.0006$ \\
\addlinespace[2pt]
IAE (cosine) & none & $0.2284{\pm}0.0277$ & $0.2293{\pm}0.0248$ & $0.2276{\pm}0.0306$ \\
 & per-dim & $0.0561{\pm}0.0004$ & $0.0668{\pm}0.0005$ & $0.0459{\pm}0.0006$ \\
 & low-freq & $\mathbf{0.0507{\pm}0.0005}$ & $0.0607{\pm}0.0006$ & $0.0412{\pm}0.0005$ \\
\addlinespace[2pt]
IAE (Legendre) & none & $0.2663{\pm}0.0530$ & $0.2635{\pm}0.0526$ & $0.2690{\pm}0.0536$ \\
 & per-dim & $0.0513{\pm}0.0007$ & $0.0616{\pm}0.0009$ & $0.0414{\pm}0.0006$ \\
 & low-freq & $\mathbf{0.0505{\pm}0.0003}$ & $0.0603{\pm}0.0002$ & $0.0411{\pm}0.0005$ \\
\midrule
GINO~\cite{Li2023Geometry} & --- & $0.0773{\pm}0.0017$ & $0.0852{\pm}0.0022$ & $0.0697{\pm}0.0011$ \\
\bottomrule
\end{tabular}
\end{table}
 
Table~\ref{tab:ex1} reports the relative prediction errors.
Normalization substantially improves both IAE variants, and low-frequency normalization has the best overall performance.
With normalization, both variants consistently outperform the MFE and GINO baselines in the overall, notched-domain, and smooth-domain evaluations.
The cosine and Legendre variants exhibit comparable accuracy, suggesting limited sensitivity to the choice of basis in this experiment.
Normalization also benefits the MFE representations, although its effect is less consistent across encoders and geometry subsets. 
These results support the effectiveness of the IAE in encoding the inputs of the variable-domain solution operator.

\subsection{Hele--Shaw flow}
\label{sec:ex2}
We next use a Hele--Shaw transmission problem to show that the IAE can directly predict interfaces even under topological changes.
Let the domain $\Omega$ be decomposed into $\Omega^-(t)$ and $\Omega^+(t)$ by $\Gamma(t)$, with mobility
\begin{equation}
 M(\xi,t)=
 \begin{cases}
 M^-,&\xi\in\Omega^-(t),\\
 M^+,&\xi\in\Omega^+(t).
 \end{cases}
 \label{eq:hsmobility}
\end{equation}
The pressure transmission and interface motion are formulated as
\begin{equation}
\begin{aligned}
\nabla \cdot (M \nabla p) &= -s &&\text{in } \Omega\setminus\Gamma(t),\\
[p]_{\Gamma(t)} &= 0,\quad [M\partial_n p]_{\Gamma(t)} = 0 &&\text{on } \Gamma(t),\\
p &= 0 &&\text{on } \partial\Omega,\\
V_n &= -M\partial_n p\big|_{\Gamma(t)},\quad \Gamma(0) = \Gamma_0.
\end{aligned}
\label{eq:hs}
\end{equation}
Here $p$ is the pressure, $s$ is the prescribed source, $n$ fixes the interface orientation, $[\cdot]_\Gamma$ denotes the corresponding jump, and $V_n$ is the normal velocity, which is single-valued due to flux continuity.
The target solution operator is
\[
 \mathcal G_{t_k}^{\mathrm{HS}}:(\Gamma_0,s)\longmapsto\Gamma(t_k).
\]
That is, each future interface is predicted directly from $(\Gamma_0,s)$.
The dataset contains $1000$ training trajectories and $387$ test trajectories; among the test trajectories, $222$ merge and $165$ do not.

\begin{table}[tbp]
\centering
\caption{Hausdorff errors for the Hele--Shaw flow. The results are reported as mean $\pm$ standard deviation over five independent experiments. Boldface marks the lowest displayed overall mean.}
\label{tab:ex2a}
\small
\begin{tabular}{lccc}
\toprule
configuration & overall $d_H$ & merge & non-merge \\
\midrule
nonlinear, $r_{\mathrm{out}}=24$ & $\mathbf{0.0110{\pm}0.0001}$ & $0.0131{\pm}0.0001$ & $0.0081{\pm}0.0001$ \\
nonlinear, $r_{\mathrm{out}}=16$ & $0.0129{\pm}0.0000$ & $0.0162{\pm}0.0001$ & $0.0083{\pm}0.0000$ \\
nonlinear, $r_{\mathrm{out}}=12$ & $0.0142{\pm}0.0001$ & $0.0184{\pm}0.0001$ & $0.0087{\pm}0.0000$ \\
linear, $r_{\mathrm{out}}=24$ & $0.0119{\pm}0.0001$ & $0.0149{\pm}0.0001$ & $0.0080{\pm}0.0001$ \\
\bottomrule
\end{tabular}
\end{table}

\begin{figure}[htbp]
\centering
\includegraphics[width=0.9\textwidth]{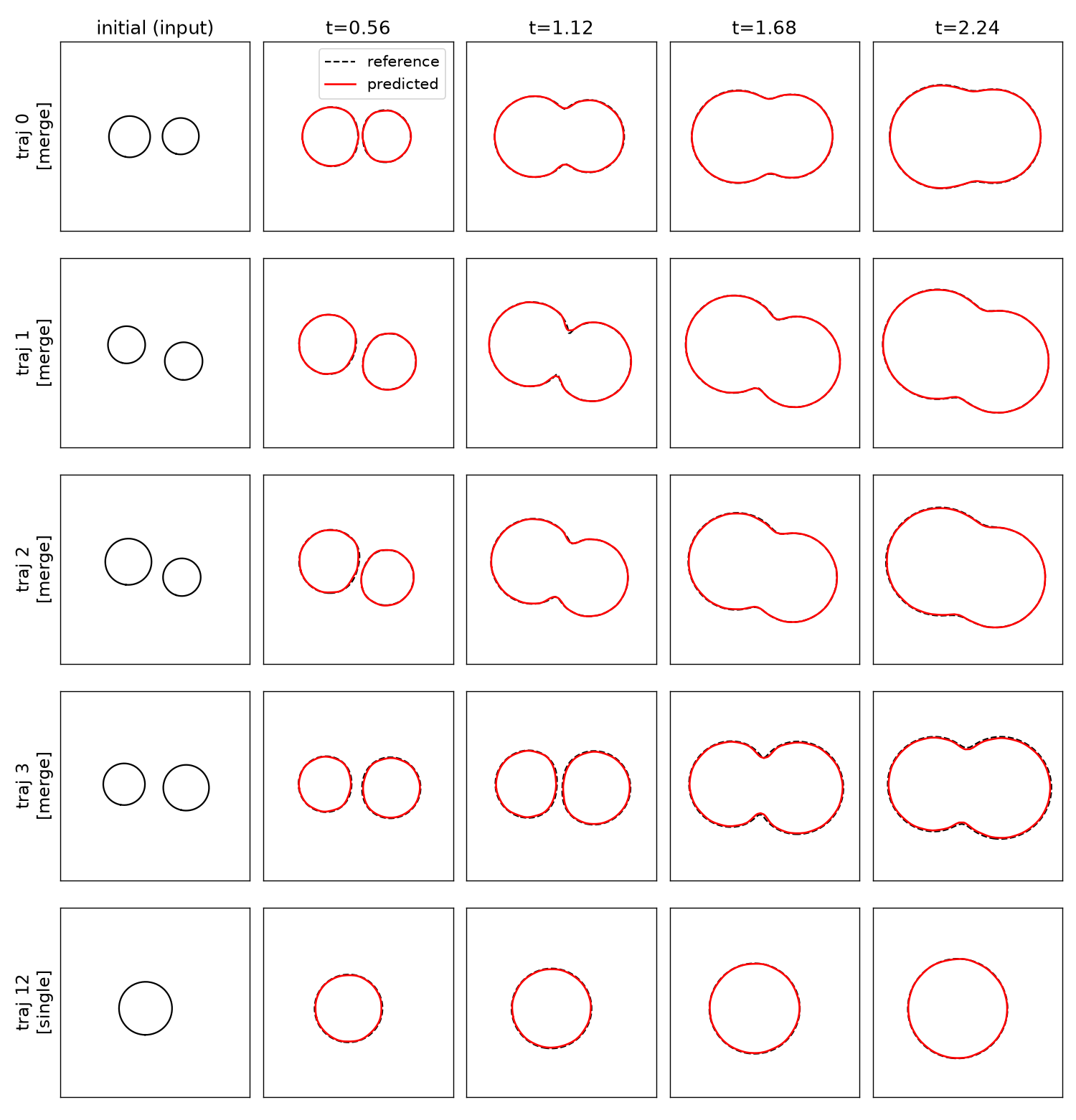}
\caption{Prediction of Hele--Shaw flow for the order-$24$ nonlinear head: four merging pairs and one single drop. The first column is the initial interface; later columns show red predictions and black dashed reference evolutions.}
\label{fig:ex2best}
\end{figure}

The initial interface is encoded by the IAE and the source by the Legendre moments in \eqref{eq:mfefield}.
Their normalized codes are processed by separate branches and combined with a time embedding:
\begin{equation}
 \begin{aligned}
 \bm a_k&=\mathcal B_{\mathrm{geom}}(z_{\mathrm{geom}}(0))
 \odot\mathcal B_{\mathrm{src}}(z_{\mathrm{src}})
 \odot\bm e_k^{\mathrm{time}}.
 \end{aligned}
 \label{eq:ex2map}
\end{equation}
We compare nonlinear and linear code heads using the same branch architectures:
\begin{equation}
 \begin{aligned}
 \text{nonlinear: }
 \widehat z=\mathcal H_{\mathrm{net}}(\bm a_k),\quad
 \text{linear: }
 \widehat z=W\bm a_k+\bm\beta.
 \end{aligned}
 \label{eq:ex2heads}
\end{equation}
The nonlinear head is evaluated at $r_{\mathrm{out}}\in\{12,16,24\}$ and the linear head at $r_{\mathrm{out}}=24$.
All configurations are trained with an auxiliary pressure loss;
the full training objective is specified in Appendix~\ref{app:numerical-hs}.
We report the Hausdorff distance $d_H$ between the unions of predicted and reference interfaces, averaged over the four future frames and all test trajectories.

Table~\ref{tab:ex2a} shows that, for the nonlinear head, increasing \(r_{\mathrm{out}}\) reduces the prediction error across all test subsets, with the greatest improvement observed for merging trajectories.
At \(r_{\mathrm{out}}=24\), the nonlinear head outperforms the linear head in the overall and merging cases, while the two heads yield comparable errors for non-merging trajectories.
This suggests that a higher-resolution output code combined with nonlinear prediction is better suited to representing such geometric changes involving interface merging.
Fig.~\ref{fig:ex2best} illustrates representative predicted trajectories, showing that the IAE can accurately capture both merging and non-merging interface dynamics.

\subsection{Stokes flow}
\label{sec:ex3}
Finally, we show that the IAE can jointly predict the evolving interface and the physical field defined on it.
We consider Stokes flow of two fluids with matched density and viscosity, which occupy the subdomains $\Omega_1(t)$ and $\Omega_2(t)$, separated by a moving interface $\Gamma(t)$.
In each phase, the velocity $\mathbf u_i$ and pressure $p_i$ satisfy the Stokes equations
\begin{equation}
 \nabla\!\cdot\mathbf T_i=\mathbf0,\qquad
 \nabla\!\cdot\mathbf u_i=0
 \quad\text{in }\Omega_i(t),\qquad
 \mathbf T_i=-p_i\mathbf I+\nabla\mathbf u_i+
 \nabla\mathbf u_i^{\!\top},\qquad i=1,2.
 \label{eq:stokes}
\end{equation}
With $\mathbf n$ pointing into $\Omega_2$, $\kappa_\Gamma=\nabla\!\cdot\mathbf n$, $\nabla_s=(\mathbf I-\mathbf n\otimes\mathbf n)\nabla$, and $[q]_\Gamma=q|_{\Omega_2}-q|_{\Omega_1}$, the interface and outer-boundary conditions are
\begin{equation}
 \begin{aligned}
 [\mathbf u]_\Gamma&=\mathbf0,
 &[\mathbf T\mathbf n]_\Gamma
 &=\frac1{\mathrm{Ca}}
 \bigl(\sigma\kappa_\Gamma\mathbf n-\nabla_s\sigma\bigr),\\
 V_n&=\mathbf u\!\cdot\!\mathbf n,
 &\mathbf u&=(y,0)\quad\text{on }\partial\Omega,\\
 \Gamma(0)&=\Gamma_0,
 &f(\cdot,0)&=f_0\quad\text{on }\Gamma_0.
 \end{aligned}
 \label{eq:stokesbc}
\end{equation}
The interface carries an insoluble surfactant concentration $f$.
The normalized Langmuir law and surface transport equation are
\begin{equation}
 \sigma(f)=
 \frac{1+E\ln(1-\chi_{\mathrm{cov}}f)}
 {1+E\ln(1-\chi_{\mathrm{cov}})},
 \qquad
 D_t^\Gamma f+f\,\nabla_s\!\cdot\mathbf u
 =\frac1{\mathrm{Pe}}\Delta_s f,
 \qquad \chi_{\mathrm{cov}}f<1.
 \label{eq:surftrans}
\end{equation}
Here $D_t^\Gamma$ and $\Delta_s$ denote the material derivative and Laplace--Beltrami operator.
The parameters are fixed at $(\mathrm{Ca},\mathrm{Pe},E,\chi_{\mathrm{cov}})=(0.5,10,0.2,0.2)$.
The target solution operator is
\[
 \mathcal G_{t_k}^{\mathrm{SS}}:(\Gamma_0,f_0)\longmapsto
 \bigl(\Gamma(t_k),f(\cdot,t_k)\bigr).
\] 
We generate $850$ training and $150$ test trajectories using the scheme of Xu et al.~\cite{Xu2006Level} with the interfacial jumps of \cite[Eq.~(3.12)]{LeVeque1997Immersed}. Sampling and implementation details are given in Appendix~\ref{app:numerical-stokes}.

Both the interface and the surfactant field are represented by order-$32$ cosine IAE codes.
A two-head MIONet predicts the two output codes jointly:
\begin{equation}
 \begin{aligned}
 \bm a_k&=\mathcal B_{\mathrm{geom}}(z_{\mathrm{geom}}(0))
 \odot\mathcal B_{\mathrm{surf}}(z_f(0))
 \odot\bm e_k^{\mathrm{time}},\\
 (\widehat z_{\mathrm{geom}}(t_k),\widehat z_f (t_k))
 &=\bigl(\mathcal H_{\mathrm{geom}}(\bm a_k),
 \mathcal H_f(\bm a_k)\bigr).
 \end{aligned}
 \label{eq:ex3map}
\end{equation}
\begin{table}[htbp]
\centering
\caption{Interface error $d_H$ and joint-state error $d_{\mathrm{gr}}$ for the three Stokes-flow interaction regimes. In both metrics, spatial coordinates are mapped to $\Bx=[0,1]^2$; the surfactant coordinate in $d_{\mathrm{gr}}$ remains on its original nondimensional scale. The error entries are reported as mean $\pm$ standard deviation over five independent experiments.}
\label{tab:ex3mode}
\small
\setlength{\tabcolsep}{6pt}
\begin{tabular}{lccccc}
\toprule
regime & \# train & \# test & interface $d_H$ & state $d_{\mathrm{gr}}$
& loop count \\
\midrule
shearing-past & 564 & 96 & $0.0042{\pm}0.0000$ & $0.0258{\pm}0.0004$ & $100\%$ \\
isolated & 187 & 32 & $0.0049{\pm}0.0001$ & $0.0413{\pm}0.0010$ & $100\%$ \\
co-translating & 99 & 22 & $0.0057{\pm}0.0003$ & $0.0367{\pm}0.0010$ & $100\%$ \\
\midrule
overall & 850 & 150 & $0.0046{\pm}0.0001$ & $0.0305{\pm}0.0003$ & $100\%$ \\
\bottomrule
\end{tabular}
\end{table}

\begin{figure}[!tbp]
\centering
\includegraphics[width=0.88\textwidth]{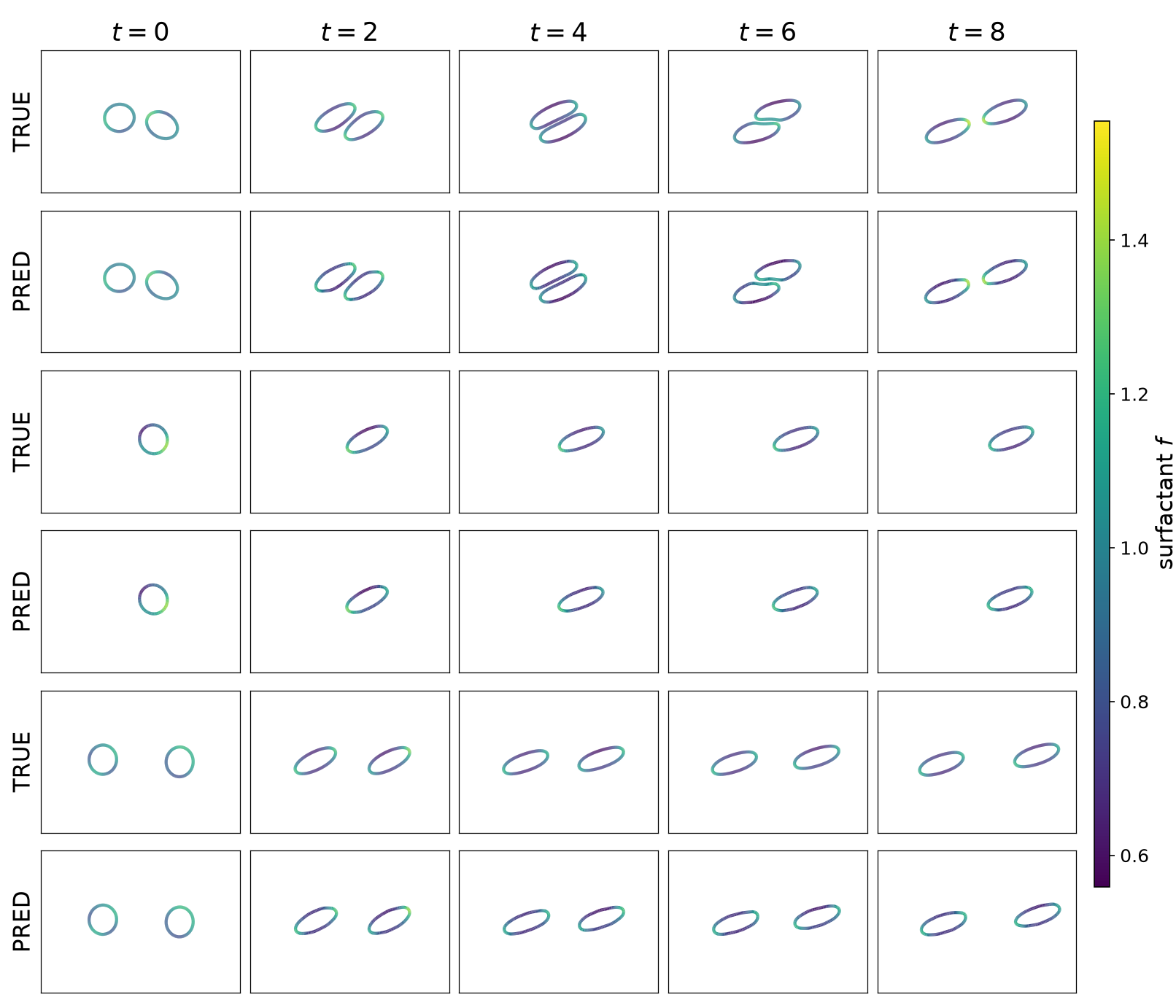}
\caption{Held-out Stokes--surfactant trajectories at $t=0,2,4,6,8$. From top to bottom, the row pairs show shearing-past, isolated, and co-translating dynamics, with the numerical reference (TRUE) above the prediction (PRED). At $t=0$, PRED is the initial state. Predictions use seed \texttt{s0}; the trajectory selection is described in Appendix~\ref{app:numerical-stokes}.}
\label{fig:ex3gallery}
\end{figure}

We evaluate the interface and joint-state predictions using $d_H$ and $d_{\mathrm{gr}}$ from \eqref{eq:hausdorff} and \eqref{eq:graphmetric}, respectively. 
Table~\ref{tab:ex3mode} shows consistently small errors across all three interaction regimes, indicating accurate prediction of both the interface evolution and the associated surfactant field.
For each of the five seeds, the predicted and reference loop counts agree on all $1158$ evaluated test frames.
This component-count statistic does not by itself certify topology preservation between output times.
The representative trajectories in Fig.~\ref{fig:ex3gallery} show close agreement with the numerical references at the displayed times.
These results demonstrate the effectiveness of the IAE for joint prediction of the evolving interface and surfactant field.

%% file: sec/conclusion.tex
This work studies operators whose inputs or outputs contain interface--function states.
We first established a general approximation principle that uses injective embeddings and convergent finite-dimensional codecs, without assuming linear state spaces or Schauder bases.
We then instantiated this principle with the IAE.
Under the stated regularity conditions, ambient approximation controls the reconstructed geometry and surface field and provides a reference-based topology certificate. 
The variable-domain Poisson problem examines cross-space prediction, while the Hele--Shaw and Stokes--surfactant problems examine within-space evolution of interfaces and their carried fields.
These numerical results support the IAE as a common representation for geometric operator learning. 

Several directions remain open.
On the theoretical side, the current results rely on pullback continuity and uniform output regularity, while the reference-based topology certificate excludes merger and breakup instants.
Establishing PDE-specific conditions for these assumptions, developing stability guarantees across topology changes, and deriving finite-data generalization bounds are therefore important next steps.
For evolutionary problems, the current time-conditioned models predict states at prescribed times without learning an explicit latent evolution law.
Coupling IAE codes with continuous-time dynamics could enforce dissipation mechanisms consistent with thermodynamic principles 
\cite{Chen2024Constructing,Zhu2026Identifiable} and potentially enable complete error analyses \cite{zhu2024error}.
The framework could also be extended to in-context learning \cite{Li2026InContext,Yang2023InContext}, using IAE codes to represent geometric states in data prompts.
Extending the framework to higher-dimensional moving surfaces, multiple interacting interfaces, strongly coupled bulk--surface systems, and more complex topology changes would help assess its scalability and clarify the need for localized or adaptive representations.

%% file: sec/numerical_details.tex
These appendices specify the numerical details for the three experiments.
\section{Shared numerical methods}
\label{app:numerical-shared}
\label{app:numerical-details}

\subsection{Discrete encoders}

Let $T$ map the physical box affinely to $\Bx=[0,1]^2$ and
set $w_B=w\circ T^{-1}$. For Poisson and Hele--Shaw, $T$ is the identity.
Unless stated otherwise, we compute tensor-product coefficients by midpoint quadrature:
\begin{equation}\label{eq:numproj}
\begin{aligned}
 a_p =\frac{p+1/2}{N},\
 \xi_{\bp}=(a_{p_1},a_{p_2}),\
 c_{\bi,N} =\frac1{N^2}\sum_{p_1,p_2=0}^{N-1}
 w_B(\xi_{\bp})\phi_{i_1}(a_{p_1})\phi_{i_2}(a_{p_2}).
\end{aligned}
\end{equation}
Here $\bi\in I_r =\{0,\ldots,r-1\}^2$. To encode geometry, we set
$w_B=\sdf_\Gamma\circ T^{-1}$.

For interface samples $(y_j,f_j)$ in physical coordinates, the Lipschitz constant $\widehat L$ in the McShane--Whitney extension is estimated with
\begin{equation}\label{eq:numlip}
 \widehat L=1.2\,\operatorname{Q}_{0.99}
 \left\{\frac{|f_j-f_l|}{\norm{y_j-y_l}_2+10^{-9}}:
 (j,l)\in\mathcal N\right\}+10^{-6}.
\end{equation}
Here $\operatorname{Q}_{0.99}$ is the empirical $0.99$-quantile. The pairs in $\mathcal N$ correspond to
endpoints of mesh boundary edges for Poisson and cyclically consecutive
entries of the concatenated interface-point array for Stokes.
This empirical $\widehat L$ need not yield exact interpolation.

\subsection{Networks and training}
Let $\operatorname{FNN}_{\rho}(d_0,\ldots,d_L)$ denote a fully connected
network with these layer widths, biases, hidden-layer activation $\rho$, and
a linear output layer. Specific parameters are given separately in each example.

Poisson inputs use the normalizations in \eqref{eq:numnorm};
Hele--Shaw inputs use coordinatewise standardization.
Hele--Shaw output codes and all Stokes codes are centered coordinatewise
and scaled blockwise.
For training codes $c^{(n)}\in\R^q$, define
\begin{equation}\label{eq:num-block-standardization}
\begin{aligned}
 \mu_j=\frac1{N_{\mathrm{tr}}}\sum_{n=1}^{N_{\mathrm{tr}}}c_j^{(n)},\
 \bar\mu=\frac1q\sum_{j=1}^q\mu_j,\
 s_{\mathrm{blk}}^2=\frac1{N_{\mathrm{tr}}q}
 \sum_{n=1}^{N_{\mathrm{tr}}}\sum_{j=1}^q(c_j^{(n)}-\bar\mu)^2,\
 \widetilde c_j=\frac{c_j-\mu_j}{s_{\mathrm{blk}}+10^{-8}}.
\end{aligned}
\end{equation}
Training statistics are block specific. The block scale uses the pooled
mean $\bar\mu$, whereas centering uses the coordinate means $\mu_j$.
The standard-deviation denominators in the code normalizations include the stabilizer $10^{-8}$.
We invert normalization before decoding.

For $N_q$ points $\xi_j$, let $h(\xi_j)$ and $\widehat h(\xi_j)$ be the reference and predicted values.
We define the value loss and weighted code loss
\begin{equation}\label{eq:numloss}
\begin{aligned}
 \mathcal L_{\mathrm{val}}
 &=\frac1{N_q}\sum_{j=1}^{N_q}
 |\widehat h(\xi_j)-h(\xi_j)|^2,\,
 \mathcal L_{\mathrm{code}}
 &=\frac{\sum_{\bi\in I_r}\omega_{\bi}^{\mathrm{out}}
 |\widehat{\widetilde c}_{\bi}-\widetilde c_{\bi}|^2}
 {\sum_{\bi\in I_r}\omega_{\bi}^{\mathrm{out}}},
\end{aligned}
\end{equation}
where $\omega_{\bi}^{\mathrm{out}}=(1+\alpha_{\mathrm{out}}\norm{\bi}_2^2)^{-1}$.
All models minimize minibatch mean losses using Adam with cosine
learning-rate decay:
\[
\begin{array}{l|cccc}
 \text{model}&\text{initial step size}&\text{batch size}&\text{updates}&\text{runs}\\\hline
 \text{Poisson MIONet}&10^{-3}&64&1.2\times10^5&5\\
 \text{GINO}&3\times10^{-4}&1&1.2\times10^5&5\\
 \text{Hele--Shaw}&10^{-3}&32&4\times10^4&5\\
 \text{Stokes}&10^{-3}&256&3\times10^4&5
\end{array}
\]
Evaluation uses the final trained models.

\subsection{Error metrics and aggregation}

The Poisson metric \eqref{eq:relL2} uses unweighted Euclidean norms
of solution values at all finite-element mesh vertices.
For nonempty finite point sets $P,Q$ and a point metric $\rho$, define the discrete Hausdorff distance
\begin{equation}\label{eq:num-discrete-hausdorff}
 H_\rho(P,Q)=\max\left\{
 \max_{p\in P}\min_{q\in Q}\rho(p,q),
 \max_{q\in Q}\min_{p\in P}\rho(p,q)\right\}.
\end{equation}
Points from all interface components are pooled without component matching.

For a reported test set, let $\mathcal I_s$ index its evaluated test samples in run $s$. Write $e_{s,i}$ for the error of sample or frame $i$ in that run.
We first average errors within each run, then report the mean and population
standard deviation of these averages across five runs:
\begin{equation}\label{eq:num-aggregation}
 E_s=\frac1{|\mathcal I_s|}\sum_{i\in\mathcal I_s}e_{s,i},
 \qquad
 \overline E=\frac15\sum_{s=0}^4E_s,
 \qquad
 \sigma_E=\left[\frac15\sum_{s=0}^4(E_s-\overline E)^2\right]^{1/2}.
\end{equation}

\section{Variable-domain Poisson problem}
\label{app:numerical-poisson}

\subsection{Domain and field sampling}
The sampled domains have the radial form in \eqref{eq:poissondomain}.
The dataset has $40000$ training domains ($19942$ notched, $20058$ smooth)
and $1000$ test domains ($489$ notched, $511$ smooth), all centered at
$\xi_0=(0.5,0.5)$. The unnotched radius is sampled at $1001$ uniformly spaced nodes
in $\tau=\theta/(2\pi)$ from a Gaussian field with mean $1$ and covariance
\[
 C_{\mathrm{rad}}(\tau,\tau')
 =(0.15)^2\exp\!\left[-\frac{2\sin^2(\pi(\tau-\tau'))}{(0.5)^2}\right].
\]
Discarding the repeated endpoint leaves $1000$ values, rescaled to maximum
absolute value $0.45$ and linearly interpolated to define $\varrho_0$.
A notch is added with probability $1/2$:
\[
\begin{aligned}
\varrho(\theta)&=\varrho_0(\theta)
-d\exp\!\left[-\left(\frac{\vartheta(\theta-\theta_0)}{w}\right)^2\right],\
d=\min\{\eta\bar\varrho_0,(\varrho_0(\theta_0)-0.045)_+\},\\
\theta_0&\sim\mathrm{Unif}[0,2\pi),\qquad
w\sim\mathrm{Unif}[0.035,0.05],\qquad
\eta\sim\mathrm{Unif}[0.45,0.72],
\end{aligned}
\]
where $\vartheta(\alpha)\in[-\pi,\pi)$ is the wrapped angle and $\bar\varrho_0$
is the mean of the rescaled radius samples. The boundary polygon uses $70$
uniformly spaced angles, supplemented by $45$ angles in
$\theta_0+[-3.5w,3.5w]$ for notched domains. Duplicate angles are removed.

Independent Gaussian fields $k,s,g$ are sampled on a $64\times64$ uniform
grid in $[0,1]^2$, with means $\mu_q$ and covariances
\[
 C_q(\xi,\xi')=\sigma_q^2
 \exp\!\left(-\frac{\norm{\xi-\xi'}_2^2}{2\ell_q^2}\right),
 \qquad q\in\{k,s,g\}.
\]
The sampling and rescaling parameters are
\begin{equation}\label{eq:poissongrf}
\begin{aligned}
(\mu_k,\sigma_k,\ell_k,L_k^{\mathrm{grid}})&=(1,0.2,0.2,4),\\
(\mu_s,\sigma_s,\ell_s,L_s^{\mathrm{grid}})&=(0,1,0.2,12),\\
(\mu_g,\sigma_g,\ell_g,L_g^{\mathrm{grid}})&=(0,0.02,0.5,0.6).
\end{aligned}
\end{equation}
Here $\sigma_q,\ell_q$ are the pre-rescaling standard deviation and length
scale. The radial and spatial Gaussian samplers use Cholesky factorization,
with diagonal jitter $10^{-10}$ added to the correlation matrix.
For each grid sample $\mathbf z_q$ with maximum
finite-difference gradient norm $G_q$, we set
\[
\mathbf z_q\leftarrow\mu_q\mathbf1+
\min\{1,L_q^{\mathrm{grid}}/G_q\}
(\mathbf z_q-\mu_q\mathbf1),\qquad
\mathbf z_k\leftarrow\max\{\mathbf z_k,0.1\},
\]
with a rescaling factor of $1$ when $G_q=0$. The resulting grid values
are interpolated using cubic splines.

\subsection{Reference solver}

Reference solutions use quadratic Lagrange elements on Gmsh meshes
with size bounds $0.0028$ and $0.030$ and local size $0.007$ near the notch.
The linear systems are solved by direct LU factorization.
The coefficient $k$ is interpolated in the same quadratic space.
Learning and evaluation use mesh-vertex solution and input-field values.

\subsection{Encodings and training}

The four representations in Eqs.~\eqref{eq:mfedgeom}--\eqref{eq:mfefield}
differ only in geometry and boundary-data encodings. Each code has $144$
coefficients, and all four representations use the same order-$12$ Legendre moments for $k$ and $s$.
IAE projects the SDF and extended boundary datum using a
$24\times24$ midpoint grid. Domain moments use piecewise-linear
interpolation of stored vertex values followed by three-point triangle quadrature;
boundary moments use linear edge interpolation and five-point
Gauss--Legendre quadrature.

The value-head MIONet in Eq.~\eqref{eq:ex1net} uses the branch and trunk networks
\[
\begin{aligned}
\mathcal B_c&=\operatorname{FNN}_{\mathrm{ReLU}}(144,256,256,256,128),
&&c\in\{\mathrm{geom},k,s,g\},\\
\mathcal T&=\operatorname{FNN}_{\mathrm{ReLU}}(2,256,256,256,128).
\end{aligned}
\]
Each representation uses the three normalizations in \eqref{eq:numnorm},
with $\alpha_{\mathrm{in}}=0.5$ for low-frequency rescaling, and minimizes
$\mathcal L_{\mathrm{val}}$ in \eqref{eq:numloss}. For each training domain, we uniformly
sample $300$ vertices from the full set of mesh vertices, including boundary vertices.
These vertices are fixed across updates and the five runs with seeds $0,\ldots,4$.

\subsection{GINO reference experiment}
For GINO~\cite{Li2023Geometry}, the input point set comprises boundary vertices and points of a $96\times96$ uniform
grid satisfying $\mathrm{sdf}<-0.006$.
The input channels are $[k,s,0,\mathrm{sdf}]$ at interior points and $[k,s,g,0]$
at boundary vertices; $k$ and $s$ are interpolated linearly from the mesh.
Input channels and solution values are standardized using statistics from training points.
The input and output graph operators use a neighborhood radius of $0.12$
and two hidden layers of width $64$. The FNO uses a $24\times24$ latent grid,
$16\times16$ Fourier modes, $48$ hidden channels, and four layers.
The model has $1.41$M parameters and uses weight decay $10^{-5}$.
Each update samples one geometry with its domain-dependent input and output point sets.
Evaluation uses all test mesh vertices and the relative $\ell^2$ metric
in \eqref{eq:relL2}.
The implementation uses the NeuralOperator library~\cite{Kossaifi2025Library}.

\section{Hele--Shaw flow}
\label{app:numerical-hs}

\subsection{Initial states and sampling}

We solve \eqref{eq:hs} on $\Omega=B_{0.45}(c_\ast)$, with
$c_\ast=(0.5,0.5)$ and $(M^-,M^+)=(1,2)$.
The initial inner phase and time-independent source are
\[
 \begin{aligned}
 \Omega^-(0)&=\bigcup_{\ell=1}^{J}B_{R_\ell}(c_\ell),\quad
 s(\xi)=\sum_{\ell=1}^{J}
 \frac{Q_\ell}{\pi r_w^2}\mathbf1_{B_{r_w}(c_\ell)}(\xi),\\
 r_w&=0.028,\quad Q_\ell\sim\mathrm{Unif}[0.035,0.065].
 \end{aligned}
\]
For single-drop proposals, we set $J=1$ and sample the radius
$R_1\sim\mathrm{Unif}[0.10,0.13]$ and center offset
$c_1-c_\ast\sim\mathrm{Unif}([-0.05,0.05]^2)$.
Two-drop proposals use
\[
 \begin{aligned}
 c_1&=c_\ast+(-d/2,\delta),
 &c_2&=c_\ast+(d/2,-\delta),\\
 R_1,R_2&\overset{\mathrm{iid}}{\sim}\mathrm{Unif}[0.085,0.11],
 &\delta&\sim\mathrm{Unif}[-0.04,0.04],
 \end{aligned}
\]
with $d\sim\mathrm{Unif}[0.24,0.30]$ for close pairs and
$d\sim\mathrm{Unif}[0.42,0.52]$ for far pairs.
The proposal probabilities for (single drop, close pair, far pair) are
$(0.20,0.60,0.20)$ for merge-augmented training and
$(0.30,0.45,0.25)$ for natural testing.
Within each configuration, the sampling distributions are identical for the two splits.
Retained trajectory counts are
\[
\begin{array}{c|ccc|c}
 &\text{single drop}&\text{merging pair}&\text{non-merging pair}&\text{total}\\\hline
 \text{training}&250&742&8&1000\\
 \text{test}&162&222&3&387
\end{array}
\]

\subsection{Reference solver}

The FEniCSx solver uses Gmsh meshes with size bounds $0.012$ and $0.024$,
a continuous quadratic pressure space $V_h$, and piecewise-constant mobility. At each time step, we solve the finite-element pressure equation corresponding to \eqref{eq:hs} for $p_h\in V_h\cap H_0^1(\Omega)$ by direct LU factorization.
We project $M_h\nabla p_h$ in $L^2$ onto continuous piecewise-linear
vector fields to obtain $\bm q_h$.
The polygon vertices $x^\ell$ advance by forward Euler:
\[
  x^{\ell+1}=x^\ell-\Delta t\,
  (\bm q_h(x^\ell)\cdot n^\ell)n^\ell,
  \qquad \Delta t=0.08,
\]
where $n^\ell$ is the unit normal at step $\ell$.
The interface is reconstructed from the union signed-distance field on a
$110\times110$ grid over $[0.05,0.95]^2$.
When the minimum gap between components is below $0.02$, a merger heuristic
adopts level-$0.02$ contours if this reduces the number of components.
Segments with fewer than $12$ vertices or a perimeter below $0.06$ are discarded;
retained reference contours are resampled at $130$ equally spaced arc-length points.
We store five frames at $t_k=0.56k$, $k=0,\ldots,4$.

\subsection{Encodings and learning}

The initial signed-distance field is encoded using cosine order $24$ on a
$48\times48$ midpoint grid. The source is encoded by order-$12$ tensor-Legendre
moments of its nodal finite-element interpolant, computed by triangle quadrature.
The branches and nonlinear head in \eqref{eq:ex2map}--\eqref{eq:ex2heads} are
\[
 \begin{aligned}
 \mathcal B_{\mathrm{geom}}&=\operatorname{FNN}_{\mathrm{ReLU}}(576,256,256,256,128),\\
 \mathcal B_{\mathrm{src}}&=\operatorname{FNN}_{\mathrm{ReLU}}(144,256,256,256,128),\\
 \mathcal H_{\mathrm{net}}&=\operatorname{FNN}_{\mathrm{ReLU}}(128,256,256,256,r_{\mathrm{out}}^2).
 \end{aligned}
\]
The four future times have separate embeddings in $\R^{128}$.
All configurations also use the auxiliary pressure head
\[
 \begin{aligned}
 \widehat p_{\mathrm{std}}(\xi,t_k)
 &=\inner{\bm a_k}{\mathcal T_p(\xi)}+\beta_p,\quad
 \mathcal T_p=\operatorname{FNN}_{\mathrm{ReLU}}(2,256,256,256,128).
 \end{aligned}
\]
The output orders and total parameter counts, including the $165249$
pressure-head parameters, are
\[
\begin{array}{c|ccc|c}
 \text{head}&\multicolumn{3}{c|}{\text{nonlinear}}&\text{linear}\\
 \hline
 r_{\mathrm{out}}&24&16&12&24\\\hline
 \text{parameters}&992193&909953&881169&753857
\end{array}
\]

For each training frame, we sample $300$ mesh vertices uniformly
and reuse them across updates and runs.
Pressure targets are standardized as $p_{\mathrm{std}}=(p-\mu_p)/(s_p+10^{-12})$,
where $\mu_p$ and $s_p$ are the mean and standard deviation
of the pressure values at these vertices, pooled across training frames.
We sample uniformly from the $4000$ future trajectory--time pairs and minimize
\begin{equation}\label{eq:hs-total-loss}
 \begin{aligned}
 \mathcal L =\mathcal L_{\mathrm{code}}+0.3\mathcal L_p,\
 \mathcal L_p =\frac1{300}\sum_{\nu=1}^{300}
 \bigl(\widehat p_{\mathrm{std}}(\xi_\nu,t_k)
       -p_{\mathrm{std}}(\xi_\nu,t_k)\bigr)^2,
 \end{aligned}
\end{equation}
with $\alpha_{\mathrm{out}}=0.02$ in \eqref{eq:numloss}.

\subsection{Contour extraction and evaluation}

The decoded approximation to the signed-distance field is evaluated on a
$130\times130$ grid over $[0.02,0.98]^2$.
Zero contours with fewer than $12$ vertices or a perimeter below $0.08$
are discarded. Each retained predicted contour is sampled at $130$
uniformly spaced vertex indices, whereas each reference contour retains its $130$
arc-length samples.
We use $H_\rho$ from \eqref{eq:num-discrete-hausdorff} with
$\rho(x,y)=\norm{x-y}_2$.

\section{Stokes--surfactant flow}
\label{app:numerical-stokes}

\subsection{Reference solver and correction diagnostics}

Reference trajectories use the level-set and immersed-interface scheme of
Xu et al.~\cite{Xu2006Level}, with interfacial jump conditions from
\cite[Eq.~(3.12)]{LeVeque1997Immersed} under the conventions of \eqref{eq:stokesbc}.
We use $\Omega=[-7,7]\times[-5,5]$, grid spacing $h=0.04$, and time step
$\Delta t=h/24$, with the physical parameters given in Section~\ref{sec:ex3}.
The outer pressure closure is $\partial_{n_{\partial\Omega}}p=0$; a collocated
grid and a zero-mean DCT-II pressure solve are used.
The Langmuir logarithm is evaluated as
$\ln\max\{1-\chi_{\mathrm{cov}}f,10^{-6}\}$.
Frames are stored at $t_k=k$, $k=0,\ldots,8$.

For trajectory filtering, let $\beta_{d,n}$ be the solver's multiplicative
surfactant-mass correction for drop $d$ at internal step $n$, and set
$b_n=\max_d|\beta_{d,n}-1|$. With $n_k$ the internal step at $t_k$, define
\[
 B_k^{\mathrm{seg}}=\max_{n_{k-1}<n\le n_k}b_n,\qquad
 B_k^{\mathrm{cum}}=\max_{n\le n_k}b_n.
\]
Both diagnostics are set to zero at $t_0$.

\subsection{Initial states and trajectory filtering}

Initial drops are ellipses
\[
 x(\theta)=c+R_\vartheta
 \begin{pmatrix}\sqrt q\cos\theta\\q^{-1/2}\sin\theta\end{pmatrix},
 \qquad q\sim\mathrm{Unif}[1,1.4],\quad
 \vartheta\sim\mathrm{Unif}[0,\pi),
\]
where $R_\vartheta$ is the planar rotation matrix. One- and two-drop proposal probabilities
are $0.2$ and $0.8$. Conditional on a two-drop proposal, the shearing-past and
co-translating regimes have probabilities $0.85$ and $0.15$, respectively.
An isolated center is uniform on $[-1.2,1.2]\times[-0.9,0.9]$.
Shearing-past proposals use
\[
\begin{gathered}
 d_y\sim\mathrm{Unif}[0.25,0.35],\quad
 t_c\sim\mathrm{Unif}[4.5,6.5],\quad x_0=\max\{d_yt_c,1.5\},\\
 c_1=(-x_0,\varepsilon d_y),\quad c_2=(x_0,-\varepsilon d_y),
 \qquad \Pr(\varepsilon=1)=0.8,\quad\Pr(\varepsilon=-1)=0.2.
\end{gathered}
\]
The approaching arrangement follows \cite[Fig.~3]{Xu2006Level}.
Co-translating proposals use $c_1=(-2.5+\eta_1,\zeta_1)$ and
$c_2=(2.5+\eta_2,\zeta_2)$, with independent
$\eta_j\sim\mathrm{Unif}[-0.25,0.25]$ and
$\zeta_j\sim\mathrm{Unif}[-0.15,0.15]$.
A proposed pair is retained only if its initial center separation is at least $2.5$.
The initial concentration, in polar angle $\theta$ about each center, is
\begin{equation}\label{eq:ex3initialf}
\begin{aligned}
 f_0(\theta)&=\operatorname{clip}_{[0.4,1.8]}
 \left(1+\sum_{m=1}^3a_m\cos(m\theta+\phi_m)\right),\\
 a_m&\sim\mathrm{Unif}[-0.3/m,0.3/m],
 \qquad \phi_m\sim\mathrm{Unif}[0,2\pi).
\end{aligned}
\end{equation}
The sampled amplitudes and phases are independent across modes and drops.

Trajectories are truncated at the first numerical overlap detected by a
zero-tolerance point-in-polygon test. Only trajectories with at least three valid
stored frames are retained. Mass-correction filtering discards the first frame with
$B_k^{\mathrm{seg}}>0.006$ and all later frames.
The threshold $0.006$ was selected from pooled training- and test-set
diagnostics and applied to both splits.
This filter removes $94$ training and $12$ test frames but no trajectories;
the maximum retained $B_k^{\mathrm{cum}}$ is $0.0057$.
The retained dataset has $850$ training and $150$ test trajectories, with regime
counts in Table~\ref{tab:ex3mode}.
All retained concentrations satisfy $\chi_{\mathrm{cov}}f<1$, as required by \eqref{eq:surftrans}.
Before encoding, we trim the last frame of trajectories with
more than three stored frames and $t_{\mathrm{last}}<8-1/2$.

\subsection{Encodings and learning}

Both code blocks use $N=96$, cosine order $r=32$, dimension $r^2=1024$,
and coordinate map
\[
T(x,y)=\left(\frac{x+7}{14},\frac{y+5}{10}\right).
\]
The two branches and two heads in \eqref{eq:ex3map} have architectures
\[
\begin{aligned}
 \mathcal B_c&=\operatorname{FNN}_{\mathrm{GELU}}(1024,256,256,256,128),\\
 \mathcal H_q&=\operatorname{FNN}_{\mathrm{GELU}}(128,256,256,256,1024).
\end{aligned}
\]
Normalization statistics are computed from training trajectory--time pairs.
We sample uniformly from the available future pairs and minimize the unmasked objective
\[
 \mathcal L=\frac1{1024}\left(
 \norm{\widehat{\widetilde z}_{\mathrm{geom}}-
 \widetilde z_{\mathrm{geom}}}_2^2+
 \norm{\widehat{\widetilde z}_f-\widetilde z_f}_2^2\right).
\]

\subsection{Contour extraction and error evaluation}

Predicted contours are extracted on a $240\times240$ grid over the physical box;
segments with more than eight vertices are retained.
Single-drop references use zero contours extracted from the native level-set grid;
two-drop references use $256$ approximately uniform arc-length samples per drop.
Predicted and reference contours are not resampled to a common point count. The surfactant cosine series is evaluated at predicted contour points.
Let $P,Q$ be the pooled predicted and reference interface point clouds,
respectively, and define
\[
 P_f=\{(x,\widehat f(x)):x\in P\},\qquad
 Q_f=\{(y,f(y)):y\in Q\}.
\]
Using the discrete Hausdorff distance defined in \eqref{eq:num-discrete-hausdorff}, we evaluate
\[
\begin{aligned}
 d_H&=H_{\rho_x}(P,Q),
 &\rho_x(x,y)&=\norm{T(x)-T(y)}_2,\\
 d_{\mathrm{gr}}&=H_{\rho_{\mathrm{gr}}}(P_f,Q_f),
 &\rho_{\mathrm{gr}}((x,s),(y,t))
 &=\max\{\rho_x(x,y),|s-t|\}.
\end{aligned}
\]
Regime frame counts are $(760,224,174)$ per run, in shearing-past,
isolated, and co-translating order.
Both distances are averaged over these $1158$ frames per run.
Figure~\ref{fig:ex3gallery} uses seed \texttt{s0}, a $260\times260$ decoding grid, and full-length trajectories $127$, $612$, and $705$. The shearing-past case has the smallest reference-loop gap; the other two were selected near the $60$th percentile of loop-shape anisotropy among full-length trajectories with matching predicted loop counts.